\documentclass{elsarticle}
\usepackage{amsfonts,amsmath,amssymb,amsthm}

\newtheorem{theorem}{Theorem}[section]
\newtheorem{corollary}[theorem]{Corollary}
\newtheorem{proposition}[theorem]{Proposition}
\newtheorem{lemma}[theorem]{Lemma}

\theoremstyle{definition}

\newtheorem{examples}[theorem]{Examples}
\newtheorem{remark}[theorem]{Remark}

\newtheorem{problem}{Problem}

\theoremstyle{remark}

\begin{document}

\begin{frontmatter}

\title{On Simpleness of Semirings and Complete Semirings}

\author[yk]{Y. Katsov}\ead{katsov@hanover.edu}
\author[tgn]{T. G. Nam}\ead{trangiangnam05@yahoo.com}
\author[jz]{J. Zumbr\"agel\fnref{fn1}}\ead{jens.zumbragel@ucd.ie}

\fntext[fn1]{The work J. Zumbr\"agel was supported by the Science
  Foundation Ireland under Grants 06/MI/006 and 08/IN.1/I1950.}

\address[yk]{%
  Department of Mathematics and Computer Science\\
  Hanover College, Hanover, IN 47243--0890, USA}
\address[tgn]{%
  Department of Mathematics\\
  Dong Thap Pedagogical University, Dong Thap, Vietnam}
\address[jz]{%
  Claude Shannon Institute, School of Mathematical Sciences\\
  University College Dublin, Belfield, Dublin 4, Ireland}

\begin{abstract}
  In this paper, among other results, there are described (complete)
  simple -- simultaneously ideal- and congruence-simple --
  endomorphism semirings of (complete) idempotent commutative monoids;
  it is shown that the concepts of simpleness, congruence-simpleness
  and ideal-simpleness for (complete) endomorphism semirings of
  projective semilattices (projective complete lattices) in the
  category of semilattices coincide iff those semilattices are finite
  distributive lattices; there are described congruence-simple
  complete hemirings and left artinian congruence-simple complete
  hemirings.   Considering the relationship between the concepts of
  `Morita equivalence' and `simpleness' in the semiring setting, we
  have obtained the following results: The ideal-simpleness,
  congruence-simpleness and simpleness of semirings are Morita
  invariant properties; A complete description of simple semirings
  containing the infinite element; The representation theorem --
  ``Double Centralizer Property'' -- for simple semirings; A complete
  description of simple semirings containing a projective minimal
  one-sided ideal; A characterization of ideal-simple semirings having
  either infinite elements or a projective minimal one-sided ideal; A
  confirmation of Conjecture of \cite{kat:thcos} and solving Problem
  3.9 of \cite{kat:ofsos} in the classes of simple semirings
  containing either infinite elements or projective minimal left
  (right) ideals, showing, respectively, that semirings of those
  classes are not perfect and the concepts of `mono-flatness' and
  `flatness' for semimodules over semirings of those classes are the
  same.  Finally, we give a complete description of ideal-simple,
  artinian additively idempotent chain semirings, as well as of
  congruence-simple, lattice-ordered semirings.
\end{abstract}



\end{frontmatter}

\section{Introduction}

As is well known, structure theories for algebras of classes/varieties
of algebras constitute an important ``classical'' area of the
sustained interest in algebraic research.  In such theories, so-called
simple algebras, \textit{i.e.}, algebras possessing only two trivial
congruences -- the identity and universal ones -- play a very
important role of ``building blocks''.  In contrast to the varieties of
groups and rings, the research on simple semirings has been just
recently started, and therefore not much on the subject is
known.  Investigating semirings and their representations, one should
undoubtedly use methods and techniques of both ring and lattice theory
as well as diverse techniques and methods of categorical and universal
algebra.  Thus, the wide variety of the algebraic techniques involved
in studying semirings, and their representations/semimodules, perhaps
explains why the research on simple semirings is still behind of that
for rings and groups (for some recent activity and results on this
subject one may consult \cite{mitfenog:cfcs},
\cite{bshhurtjankepka:scs}, \cite{monico:ofcss}, \cite{bashkepka:css},
\cite{kalakepka:anofgiscs}, \cite{zumbr:cofcsswz},
\cite{jezkepkamaroti:tesoas}, and \cite{knt:mosssparp}).

At any rate, this paper concerns the ideal- and congruence-simpleness
-- in a semiring setting, these two notions of simpleness are not the
same (see Examples 3.7 below) and should be differed -- for some
classes of semirings, as well as the ideal- and congruence-simpleness
for complete semirings.  The complete semirings originally appeared on
the ``arena'' while considering both the generalization of classical
formal languages to formal power series with coefficients in semirings
and automata theory with multiplicities in semirings and, therefore,
requiring the more general concepts of `infinite summability' for
semiring elements (see, for example, \cite{eilenberg:alm},
\cite{weinert:gsos}, \cite{heb:eatusmaohuh}, \cite{ka:olics}, and
\cite{hebwei:sataaics}).

The paper is organized as follows.  In Section 2, for the reader's
convenience, there are included all subsequently necessary notions and
facts on semirings and semimodules.  Section 3, among other results,
contains two main results of the paper -- Theorem~3.3, describing
simple (\textit{i.e.}, simultaneously ideal- and congruence-simple)
endomorphism semirings of idempotent commutative monoids, and its
``complete'' analog -- Theorem~3.6.   We also show (Corollaries 3.8 and
3.9) that the concepts of simpleness, congruence-simpleness and
ideal-simpleness for endomorphism semirings (complete endomorphism
semirings) of projective semilattices (projective complete lattices)
in the category of semilattices with zero coincide iff those
semilattices are finite distributive lattices.  Theorem~4.6 and
Corollary 4.7 of Section 4, describing all congruence-simple complete
hemirings and left artinian congruence-simple complete hemirings,
respectively, are among the main results of the paper, too.

In Section 5, considering the relationship between the concepts of
`Morita equivalence' and `simpleness' in the semiring setting, there
have been established the following main results of the paper: The
ideal-simpleness, congruence-simpleness and simpleness of semirings
are Morita invariant properties (Theorem~5.6); A complete description
of simple semirings containing the infinite element (Theorem~5.7); The
representation theorem --``Double Centralizer Property''-- for simple
semirings (Theorem~5.10); A complete description of simple semirings
containing a projective minimal one-sided ideal (Theorem~5.11); A
characterization of ideal-simple semirings having either infinite
elements or a projective minimal one-sided ideal (Theorem~5.12 and
Corollary~5.13); A confirmation of Conjecture of \cite{kat:thcos} and
solving Problem~3.9 of \cite{kat:ofsos} in the classes of simple
semirings containing either infinite elements or projective minimal
left (right) ideals, showing, respectively, that semirings of those
classes are not perfect (Theorem~5.15 and Corollary 5.16) and the
concepts of `mono-flatness' and `flatness' for semimodules over
semirings of those classes coincide (Theorem~5.17).

Section 6 contains two more central results of the paper: A complete
description of ideal-simple, artinian additively idempotent chain
semirings (Theorem~6.4 and Remark~6.6); And a complete description of
congruence-simple, lattice-ordered semirings (Theorem~6.7 and
Corollary 6.8).

Finally, in the course of the paper, there have been stated several,
in our view interesting and promising, problems; also, all notions and
facts of categorical algebra, used here without any comments, can be
found in \cite{macl:cwm}, and for notions and facts from semiring
theory, universal algebra and lattice theory, we refer to
\cite{golan:sata}, \cite{gratzer:ua} and \cite{birkhoff:lathe},
respectively.

\section{Preliminaries}

Recall \cite{golan:sata} that a \textit{hemiring} is an algebra
$(R,+,\cdot ,0)$ such that the following conditions are
satisfied:
\begin{enumerate}[\ \ (1)\ ]
\item $(R,+,0)$ is a commutative monoid with identity element $0$;
\item $(R,\cdot)$ is a semigroup;
\item Multiplication distributes over addition from either side;
\item $0r=0=r0$ for all $r\in R$.
\end{enumerate}

Then, a hemiring $(R,+,\cdot ,0)$ is called a \textit{semiring} if its
multiplicative reduct $(R,\cdot )$ is actually a monoid $(R,\cdot ,1)$
with the identity $1$.  A \textit{proper} hemiring is a hemiring that
is not a ring.

As usual, a \textit{left} $R$-\textit{semimodule} over a hemiring $R$
is a commutative monoid $(M,+,0_{M})$ together with a scalar
multiplication $(r,m)\mapsto rm$ from $R\times M$ to $M$ which
satisfies the following identities for all $r,r'\in R$ and
$m,m'\in M$:
\begin{enumerate}[\ \ (1)\ ]
\item $(rr')m=r(r'm)$;
\item $r(m+m')=rm+rm'$;
\item $(r+r')m=rm+r'm$;
\item $r0_{M}=0_{M}=0m$.
\end{enumerate}

\textit{Right semimodules} over a hemiring $R$ and homomorphisms
between semimodules are defined in the standard manner.  Considering
semimodules over a semiring $R$, we always presume that they are
unital, \textit{i.e.}, $1m=m$ for all $m\in M$.  And, from now on, let
$\mathcal{M} _R$ and $_R\mathcal{M}$ denote the categories of
right and left semimodules, respectively, over a hemiring $R$.  As
usual (see, for example, \cite[Chapter 17]{golan:sata}), in the
category $_R\mathcal{M}$, a \textit{free} (left) semimodule
$\sum_{i\in I}R_{i}$, $R_{i}\cong {}_RR$, $i\in I$, with a basis set
$I$ is a direct sum (a coproduct) of $|I|$ copies of the regular
semimodule $_RR$.  And a \textit{projective} left semimodule in
$_R\mathcal{M}$ is just a retract of a free semimodule.  A
semimodule $ _RM $ is \textit{finitely generated} iff it is a
homomorphic image of a free semimodule with a finite basis set.

We need to recall natural extensions of the well known for rings and
modules notions to a context of semirings and semimodules.  Thus, the
notion of a \textit{superfluous} (or \textit{small}) subsemimodule: A
subsemimodule $S\subseteq M$ is superfluous (written $S\subseteq _sM$)
if $S+N=M \Rightarrow N=M$, for any subsemimodule $N\subseteq M$.
Then, taking into consideration Proposition~5 and Definition~4 of
\cite{tuyenthang:oss}, for any semimodule $_RA\in|_R\mathcal{M}|$, we
have the subsemimodule ${\rm Rad}(A)$ of $_RA$, the \textit{radical}
of $_RA$, defined by ${\rm Rad}(A) := \sum\, \{S\mid S\subseteq_sA\} =
{\bigcap\, \{M\mid M\text{ is a maximal subsemimodule of }A\,\}}$.
One can also extend the notions (as well as results involving them) of
\textit{Descending Chain Condition} and \textit{artinian module} of
the theory of modules over rings to a context of semimodules over
semirings in an obvious fashion (see, \textit{e.g.},
\cite{knt:ossss}).

Following \cite{gw:oeics} (see also \cite{golan:sata}), a left
semimodule $M$ over a hemiring $R$ is called \textit{complete} if and
only if for every index set $\Omega $ and for every family $\{m_i \mid
i\in \Omega\}$ of elements of $M$ we can define an element $\sum_{i\in
  \Omega }m_{i}$ of $M$ such that the following conditions are
satisfied:
\begin{enumerate}[\ \ (1)\ ]
\item $\sum_{i \in \emptyset}m_i= 0_M$,
\item $\sum_{i\in \{1\}}m_i = m_1$,
\item $\sum_{i\in \{1, 2\}m_i} = m_1 + m_2$,
\item If $\Omega = \bigcup _{j\in \Lambda }\Omega _{i}$ is a partition of
  $\Omega $ into the disjoint union of nonempty subsets then
  \begin{equation*}
    \sum_{i\in \Omega }m_{i} = 
    \sum_{j\in \Lambda }\Big(\sum_{i\in \Omega _{j}}m_{i}\Big)\:,
  \end{equation*}
\item If $r\in R$ and $\{m_i \mid i\in \Omega\}\subseteq M$, then
  $r \big(\sum_{i\in \Omega}m_i\big) = \sum_{i\in \Omega}rm_i$.
\end{enumerate}

As an immediate consequence of this definition, one gets that if $\pi
$ is a permutation of $\Omega $ then
\begin{equation*}
  \sum_{i\in \Omega }m_{i} = \sum_{j\in \Omega } 
  \Big(\sum_{i\in \{\pi(j)\}}m_{i} \Big) = \sum_{j\in \Omega }m_{\pi (j)} \:.
\end{equation*}

Complete right $R$-semimodules are defined analogously.  For a
complete left (right) $R$-semimodule $M$, we always have $\sum
_{i\in \Omega }0_{M} = \sum _{i\in \Omega }0_R0_{M} = 0_R\big(\sum
_{i\in \Omega }0_{M}\big) = 0_{M}$.

A hemiring is complete if and only if it is complete as a left and
right semimodule over itself.  The Boolean semifield
$\mathbf{B}=\{0,1\}$ -- an idempotent two element semiring -- is a
complete semiring if we define $\sum_{i\in \Omega}r_{i}=0$ iff
$r_{i}=0$ for all $i\in \Omega $, and to be $1$ otherwise.

If $M$ and $N$ are complete left $R$-semimodules then a
$R$-homomorphism $ \alpha: M\longrightarrow N$ is \textit{complete} if
and only if it satisfies the condition that $\alpha (\sum _{i\in
  \Omega}m_{i}) = \sum _{i\in \Omega}\alpha (m_{i})$ for every index
set $\Omega$.  Complete homomorphism of complete hemirings (semirings)
are defined analogously.  We will denote the set of all complete
$R$-homomorphisms from $M$ to $N$ by ${\rm CHom}_R(M,N)$ .  Similarly,
we denote the set of all complete $R$-endomorphisms of a complete left
$R$-semimodule $M$ by ${\rm CEnd}_R(M)$.  Then ${\rm {\rm CEnd}}_R(M)$
can be made a complete semiring with the infinite summation to be
defined by $ (\sum_{i\in \Omega }\alpha )(m) = \sum_{i\in
  \Omega}\alpha_{i}(m)$ for every set $\{\alpha_i\mid i\in \Omega
\}\subseteq {\rm CEnd}_R(M)$ and $m\in M$.

Following \cite{bshhurtjankepka:scs}, a hemiring $R$ is
\textit{congruence-simple} if the diagonal, $\vartriangle _R$, and
universal, $ R^2$, congruences are the only congruences on $R$; and
$R$ is \textit{ideal-simple} if $0$ and $R$ are its only ideals.  A
hemiring $R$ is said to be \textit{simple} if it is simultaneously
congruence- and ideal-simple.

Any hemiring $R$ clearly can be considered as a subsemiring of the
endomorphism semiring $\mathbf{E}_R$ of its additive reduct $(R,+,0)$,
and, by\ \cite[Proposition 3.1]{zumbr:cofcsswz}, the
congruence-simpleness of a proper hemiring $R$\ implies that the
reduct $(R,+,0)$ is an idempotent monoid, \textit{i.e.}, $(R,+,0)$ is,
in fact, an upper semilattice.  In light of these observations, it is
reasonable (and we will do that in the next section) first to consider
the concepts of simpleness introduced above for the endomorphism
semiring\ $\mathbf{E}_M$ of a semilattice -- an idempotent commutative
monoid -- $(M,+,0)$.  Recall (see \cite[Definition
1.6]{zumbr:cofcsswz}) that a subhemiring $S\subseteq \mathbf{E}_M$ of
the endomorphism semiring\ $\mathbf{E}_M$ is \textit{dense} if it
contains for all $a,b\in M$ the endomorphisms $e_{a,b}\in
\mathbf{E}_M$ defined by
\begin{equation*}
  e_{a,b}(x) :=
  \begin{cases}
    0 &\text{if } x+a=a,\\
    b &\text{otherwise,}
  \end{cases}\quad
  \text{for any }x\in M\:.
\end{equation*}

We will need the following results.

\begin{lemma}[{\cite[Lemma~2.2]{zumbr:cofcsswz}}]
  For any $a,b,c,d\in M$ and $f\in\mathbf{E}_M$ we have 
  $f\circ e_{a,b} = e_{a,f(b)}$ and
  \begin{equation*}
    e_{c,d}\circ f\circ e_{a,b} :=
    \begin{cases}
      0 &\text{if } f(b)\le c,\\
      e_{a,d} &\text{otherwise}.
    \end{cases}
  \end{equation*}
\end{lemma}

\begin{theorem}[{\cite[Theorem~1.7]{zumbr:cofcsswz}}]
  A proper finite hemiring $R$ is congruence-simple iff $|R|\le 2$, or
  $R$ is isomorphic to a dense subhemiring $S\subseteq\mathbf{E}_M$ of
  the endomorphism semiring $\mathbf{E}_M$ of a finite semilattice
  $(M,+,0)$.
\end{theorem}

\section{Simpleness of Endomorphism Hemirings}

Now let $\mathbf{E}_M$ be the endomorphism hemiring of an idempotent
commutative monoid (a semilattice with zero) $(M,+,0)$, $\mathbf{F}_M
:= \langle e_{a,b} \mid a,b\in M \rangle\subseteq \mathbf{E}_M$ the
submonoid of $(\mathbf{E}_M,+,0)$ generated by all endomorphisms
$e_{a,b}$, $a,b\in M$, and $\mathbf{Fr}_M := \{f\in\mathbf{E}_M \mid
|f(M)| < \infty \}\subseteq \mathbf{E}_M$ the submonoid of
$(\mathbf{E}_M,+)$ consisting of all endomorphisms of the semilattice
$M$ having finite ranges.  It is obvious that $\mathbf{Fr}_M$ is an
ideal of $\mathbf{E}_M$, and $\mathbf{F} _M\subseteq \mathbf{Fr}_M$;
moreover, the following observations are true.

\begin{lemma}
  Let $(M,+,0)$ be an idempotent commutative monoid such that $|\{
  {x\in M} \mid f(x)\le a\}| < \infty$ for any $a\in M$ and
  $f\in\mathbf{E}_M$.  Then, $\mathbf{F}_M$ is an ideal of
  $\mathbf{E}_M$.  In particular, if $M$ is finite, $\mathbf{F}_M$ is
  an ideal of $\mathbf{E}_M$.
\end{lemma}

\begin{proof}
  By Lemma~2.1, it is enough to show that $\mathbf{F}_M$ is a right
  ideal of $\mathbf{E}_M$.  Indeed, for any $a,b,x\in M$
  and $f\in \mathbf{E}_M$,
  \begin{equation*}
    e_{a,b}\,f(x) := \begin{cases}
      0 &\text{if }f(x)\le a\:,\\
      b &\text{otherwise}\:.
    \end{cases}
  \end{equation*}
  Then, considering the set $\{x_1, x_2, \ldots, x_n\} := \{ x\in M
  \mid f(x)\le a\}$ and the element $ c := x_1 + x_2 + \ldots + x_n$,
  one can easily see that $e_{a,b}\,f = e_{c,b}$.
\end{proof}

\begin{lemma}
  For a subsemiring $S\subseteq \mathbf{E}_M$ of the endomorphism
  semiring $\mathbf{E}_M$ of an idempotent commutative monoid
  $(M,+,0)$, there holds:
  \begin{enumerate}[\ \ (i)\ ]
  \item If $S$ is ideal-simple and $S\cap\mathbf{Fr}_M\ne 0$, then
    $M$ is finite;
  \item If $S$ is ideal-simple and $S\cap\mathbf{F}_M\ne 0$, then $M$
    is a finite distributive lattice.
  \end{enumerate}
\end{lemma}

\begin{proof}
  (i).  Consider the ideal $\langle f\rangle$ of $S$ generated by a
  nonzero endomorphism $f\in S\cap \mathbf{Fr}_M$.  Since $S$ is
  ideal-simple and $1_{\mathbf{E}_M} = {\rm id}_M\in S\subseteq
  \mathbf{E}_M$, there are some endomorphisms $g_1, h_1, g_2,
  h_2,\ldots, g_n, h_n\in S$ such that ${\rm id}_M = g_1fh_1 + g_2fh_2
  + \ldots + g_nfh_n$.  Whence and since $f\in\mathbf{Fr}_M$, one has
  ${\rm id}_M\in\mathbf{Fr}_M$ and $|M|<\infty$.

  (ii).  As $\mathbf{F}_M\subseteq\mathbf{Fr}_M$, from (i) follows
  that $ |M|<\infty $.  Again, using the ideal-simpleness of $S$, one
  has that the ideal $\langle f\rangle$ of $S$ generated by a nonzero
  endomorphism $f\in S\cap \mathbf{F} _M$ contains ${\rm id}_M$, and
  therefore, by Lemma~3.1, ${\mathbf{F}_M = \mathbf{E}_M}$.  Then, using
  \cite[Proposition 4.9 and Remark~4.10]{zumbr:cofcsswz}, one
  concludes that the semilattice $M$ is, in fact, a finite
  distributive lattice.
\end{proof}

The next result describes the simple endomorphism semirings\ of idempotent
commutative monoids.

\begin{theorem}
  The following conditions for the endomorphism semiring
  $\mathbf{E}_M$ of an idempotent commutative monoid $(M,+,0)$ are
  equivalent:
  \begin{enumerate}[\ \ (i)\ ]
  \item $\mathbf{E}_M$ is simple;
  \item $\mathbf{E}_M$ is ideal-simple;
  \item The semilattice $M$ is a finite distributive lattice.
  \end{enumerate}
\end{theorem}

\begin{proof}
  (i) $\Longrightarrow$ (ii).  It is obvious.

  (ii) $\Longrightarrow$ (iii) follows from Lemma~3.2.

  (iii) $\Longrightarrow$ (i).  For $M$ is a finite distributive
  lattice, by \cite[Proposition 4.9 and Remark~4.10]{zumbr:cofcsswz},
  $\mathbf{F}_M = \mathbf{E}_M$, and therefore, by Theorem~2.2,
  $\mathbf{E}_M$ is a congruence-simple semiring.  So, we need only to
  show that $\mathbf{E}_M$ is also ideal-simple.  Indeed, let
  $I\subseteq\mathbf{E}_M$ be a nonzero ideal of $\mathbf{E}_M$,
  $0\neq f\in I$, and $f(m)\neq 0$ for some $m\in M$.  Then, by
  Lemma~2.1, $e_{a,b} = e_{0,b}\,f\,e_{a,m}\in I$ for any $a,b\in M$,
  and hence, $\mathbf{E}_M = \mathbf{F}_M\subseteq I$, \textit{i.e.},
  $\mathbf{E}_M=I$.
\end{proof}

Since the simpleness of additively idempotent semirings is one of our
central interests in this paper, the following facts about additively
idempotent semirings should be mentioned: By \cite[Proposition
3.4]{gw:oeics} (or \cite[Proposition 23.5]{golan:sata}), an additively
idempotent semiring $R$ can be embedded in a finitary complete
semiring~$S$ which additive reduct $(S,+,0)$ is a complete semimodule
over the Boolean semifield $\mathbf{B} = \{0,1\}$, \textit{i.e}., with
the semilattice $(S,+,0)$ being a complete lattice (see, for example,
\cite[Proposition 23.2]{golan:sata}); And there exists the natural
complete semiring injection from $S$ to the complete endomorphism
semiring ${\rm CEnd}_{\mathbf{B}}(S)$.  Thus, it is clear that an
additively idempotent semiring $R$ can be considered as a subsemiring
of a complete endomorphism semiring of a complete lattice; and
therefore, it is quite natural that our next results concern the
simpleness of the complete endomorphism semiring ${\rm
  CEnd}_{\mathbf{B}}(M)$ of a complete lattice $M$.

So, let $M$ be a complete lattice, which by \cite[Proposition 23.2]
{golan:sata}, for example, can be treated as a complete semimodule
over the Boolean semifield $\mathbf{B} = \{0,1\}$.  The following
observations are proved to be useful.\medskip

\begin{lemma}
  Let $M$ be a complete lattice.  Then the following statements hold:
  \begin{enumerate}[\ \ (i)\ ]
  \item $e_{a,b}\in {\rm CEnd}_{\mathbf{B}}(M)$ for all $a,b\in M$;
  \item for any $a,b\in M$ and $f\in {\rm CEnd}_{\mathbf{B}}(M)$,
    there exists an element $c\in M$ such that $e_{a,b}\, f = e_{c,b}$\,.
  \end{enumerate}
\end{lemma}

\begin{proof}
  (i).  Let $a,b\in M$ and $\{m_i\mid i\in\Omega\}$ be a family of
  elements of $M$.  It is easy to see that
  \begin{equation*}
    \bigvee_{i\in \Omega }m_{i}\le a\ \Longleftrightarrow\ \forall i\in \Omega: 
    m_{i}\le a \:.
  \end{equation*}
  Then, $e_{a,b}(\bigvee_{i\in\Omega}m_i) = \bigvee_{i\in \Omega}
  e_{a,b}(m_{i})$ and, hence, $e_{a,b}\in {\rm CEnd}_{\mathbf{B}}(M)$.

  (ii).  For any $a,b\in M$ and $f\in{\rm CEnd}_{\mathbf{B}}(M)$, we have 
  \begin{equation*}
    e_{a,b}\,f(x) := 
    \begin{cases}
      0 &\text{if }f(x)\le a\:,\\
      b &\text{otherwise}\:.
    \end{cases}
  \end{equation*}
  Let $X := \{x\in M\mid f(x)\le a\}$ and $c = \bigvee_{x\in X}x$.
  Then, $f(c) = f(\bigvee_{x\in X}x) = \bigvee_{x\in X}f(x)\le a$;
  hence, $c\in X$ and $x\le c\Leftrightarrow f(x)\le a$, and it is
  clear that $e_{a,b}\,f = e_{c,b}$.
\end{proof}

\begin{corollary}
  For any complete lattice $M$, the following statements hold:
  \begin{enumerate}[\ \ (i)\ ]
  \item $\mathbf{F}_M$ is an ideal in ${\rm CEnd}_{\mathbf{B}}(M)$;
  \item Any dense subhemiring of ${\rm CEnd}_{\mathbf{B}}(M)$ is
    congruence-simple.  In particular, ${\rm CEnd}_{\mathbf{B}}(M)$ is
    a congruence-simple semiring.
  \end{enumerate}
\end{corollary}

\begin{proof}
  (i). Since ${\rm CEnd}_{\mathbf{B}}(M)\subseteq \mathbf{E}_M$, the
  statement immediately follows from Lemmas 2.1 and 3.4\,(ii).

  (ii). For any dense subhemiring of ${\rm CEnd}_{\mathbf{B}}(M)$ is
  also a dense subsemiring of $\mathbf{E}_M$ and the complete lattice
  $M$ obviously contains the join-absorbing element $\infty :=
  \bigvee_{x\in M}x$, the statement right away follows from
  \cite[Proposition 2.3]{zumbr:cofcsswz}.
\end{proof}

The following result is a ``complete'' analog of Theorem~3.3 and
describes the simple complete endomorphism semirings of complete
lattices.

\begin{theorem}
  For any complete lattice $M$, the following are equivalent:
  \begin{enumerate}[\ \ (i)\ ]
  \item ${\rm CEnd}_{\mathbf{B}}(M)$ is simple;
  \item ${\rm CEnd}_{\mathbf{B}}(M)$ is ideal-simple;
  \item $M$ is a finite distributive lattice.
  \end{enumerate}
\end{theorem}

\begin{proof}
  (i) $\Longrightarrow$ (ii).  It is obvious.

  (ii) $\Longrightarrow$ (iii).  For ${\rm CEnd}_{\mathbf{B}}(M)$ is a
  subsemiring of $\mathbf{E}_M$, by Corollary 3.5\,(i) (or Lemma
  3.4\,(i)) $\mathbf{F}_M\subseteq {\rm CEnd}_{\mathbf{B}}(M)$, and
  hence, by Lemma~3.2\,(ii) $M$ is a finite distributive lattice.

  (iii) $\Longrightarrow$ (i).  If $M$ is a finite nonzero
  distributive lattice, by Theorem~3.3 $\mathbf{E}_M$ is a simple
  semiring.  Since $0 \neq e_{0,m}\in \mathbf{F}_M$ for all nonzero
  $m\in M$, by Lemma~3.1 $\mathbf{F}_M$ is a nonzero ideal of
  $\mathbf{E}_M$ and, hence, $\mathbf{E}_M=\mathbf{F}_M$.  Then, using
  Corollary 3.5\,(i) and ${\rm CEnd}_{\mathbf{B}}(M)\subseteq
  \mathbf{E}_M = \mathbf{F}_M\subseteq {\rm CEnd}_{\mathbf{B}}(M)$ one
  obtains that ${\rm CEnd}_{\mathbf{B}}(M) = \mathbf{E}_M$ is simple.
\end{proof}

\begin{examples}
  Congruence-simpleness of semirings does not imply their
  ideal-simpleness.  Indeed, the following examples of infinite and
  finite semirings illustrate this situation:\smallskip

  a) Let $P(X)$ be the distributive complete lattice of all subsets of
  an infinite set $X$, and ${\rm CEnd}_{\mathbf{B}}
  ({\mathcal{P}(X)})$ the complete endomorphism semiring of the
  compete lattice $(P(X),\cup)$.  Then, by Theorem~3.6 the infinite
  semiring ${\rm CEnd}_{\mathbf{B}} ({\mathcal{P}(X)})$ is not
  ideal-simple, however, by Corollary 3.5 (2) ${\rm CEnd}_{\mathbf{B}}
  ({\mathcal{ P}(X)})$ is congruence-simple.\smallskip

  b) Let $\mathbf{E}_M$ be the endomorphism semiring of a finite
  semilattice $(M,+,0)$ that is not a distributive lattice.  Then, by
  Theorems 2.2 and 3.3, the finite semiring $\mathbf{E}_M$ is an
  example of a congruence-simple, but not ideal-simple, finite
  semiring.\smallskip
  
  c) Obviously, the semifield $\mathbb{R}^+$of all nonnegative real
  numbers is an example of an ideal-simple, but not congruence-simple
  (all positive real numbers constitute a congruence class of the
  nontrivial congruence), semiring.
\end{examples}

One can clearly observe that the category of semilattices with zero
coincides with the category $_{\mathbf{B}}\mathcal{M}$ of semimodules
over the Boolean semifield $\mathbf{B}$.  Then, in light of this
observation and Examples 3.7, the following result, describing
projective $\mathbf{B}$-semimodules with simple endomorphism
semirings, is of interest.

\begin{corollary}
  For a projective $\mathbf{B}$-semimodule $M\in
  |_{\mathbf{B}}\mathcal{M}|$, the following conditions are
  equivalent:
  \begin{enumerate}[\ \ (i)\ ]
  \item $\mathbf{E}_M$ is ideal-simple;
  \item $\mathbf{E}_M$ is simple;
  \item $\mathbf{E}_M$ is congruence-simple;
  \item the semilattice $M$ is a finite distributive lattice.
  \end{enumerate}
\end{corollary}

\begin{proof}
  First, as was shown in \cite[Theorem~5.3]{horn-kim:cos}, a
  semilattice $M\in |_{\mathbf{B}}\mathcal{M}|$ is projective iff $M$
  is a distributive lattice such that $|\{x\in M \mid x\le a\}| <
  \infty$ for any $a\in M$.  Using this fact and Theorem~3.3, we need
  only to demonstrate the implication (iii) $\Longrightarrow$ (iv).

  Indeed, consider the congruence $\tau$ on $\mathbf{E}_M$ defined
  for ${f,g\in\mathbf{E}_M}$ by
  \begin{equation*}
    f\,\tau\, g \ \Longleftrightarrow\ \exists\, a\in M\ 
    \forall\, x\in M:\ f(x)+a=g(x)+a \:. 
  \end{equation*}
  If $|M|\ge 2$, one easily sees that there exists $m\in M$ such that
  $0\ne e_{0,m}\in \mathbf{E}_M$ and $(e_{0,m},0)\in \tau $.  From
  the latter and the congruence-simpleness of the semiring
  $\mathbf{E}_M$, one has that $\tau = \mathbf{E}_M\times
  \mathbf{E}_M$, and hence, ${\rm id}_M\,\tau\,0$, \textit{i.e.}, there
  exists an element $a\in M$ such that ${\rm id}_M(x)+a=a$ for any $x\in
  M$, \textit{i.e.}, $x\le a$ for any $x\in M$.  For $M$ is a
  projective $\mathbf{B}$-semimodule, from the fact mentioned above
  one gets that $|M|<\infty$.
\end{proof}

Taking into consideration Lemma~3.4\,(i) and repeating the proof of
Corollary~3.8 verbatim, we immediately get the following ``complete''
analog of the latter.

\begin{corollary}
  For a complete, projective $\mathbf{B}$-semimodule $M\in
  |_{\mathbf{B}}\mathcal{M}|$, the following conditions are
  equivalent:
  \begin{enumerate}[\ \ (i)\ ]
  \item ${\rm CEnd}_{\mathbf{B}}(M)$ is ideal-simple;
  \item ${\rm CEnd}_{\mathbf{B}}(M)$ is simple;
  \item ${\rm CEnd}_{\mathbf{B}}(M)$ is congruence-simple;
  \item $M$ is a finite distributive lattice.
  \end{enumerate}
\end{corollary}

In light of Corollaries 3.8 and 3.9, we conclude this section with the
following interesting open

\begin{problem}
  Describe all $\mathbf{B}$-semimodules $M\in
  |_{\mathbf{B}}\mathcal{M}|$ for which the conditions (i), (ii), and
  (iii) of Corollaries 3.8 and 3.9 are equivalent.
\end{problem}

\section{Congruence-simpleness of complete hemirings}

Considering a complete left $R$-semimodule $(M,+,\Sigma)$ over a
complete hemiring $(R,+,\cdot,0,\Sigma)$ with $\Sigma$ symbolizing the
`\textit{summation}' in complete semimodules and hemirings, it is
natural to call $M$ a $dR$\textit{-semimodule} if $(\sum_{i\in
  \Omega}r_{i})m = \sum_{i\in \Omega}r_im$ for all $m\in M$ and
every family $\{r_i\mid i\in \Omega\}\subseteq R$ of elements of~$R$.
Also, a left $dR$-semimodule $M$ is called \textit{s-simple} iff
$R\,M\ne 0$, and $0$ is the only proper subsemimodule of $M$; a
subsemimodule $N$ of a complete left $R$-semimodule $M$ is
\textit{complete} iff $\Sigma _{i\in \Omega}m_i\in N$ for every family
$\{m_i\mid i\in \Omega\}$ of elements of $N$.  The following
observation is almost obvious.

\begin{lemma}
  For a left $dR$-semimodule $M$ over a complete hemiring $R$, the
  following statements are equivalent:
  \begin{enumerate}[\ \ (i)\ ]
  \item $M$ is s-simple;
  \item $M\neq 0$, and $M=Rm$ for any nonzero $m\in M$;
  \item $R\,M\ne 0$, and $0$ is the only proper complete subsemimodule
    of $M$.
  \end{enumerate}
\end{lemma}

\begin{proof}
  (i) $\Longrightarrow$ (ii).  It is clear that $M\ne 0$.  Consider
  the subsemimodule $N := \{m\in M \mid Rm=0 \}$ of $M$.  It is
  obvious that $N\ne M$, and hence, because $M$ is s-simple, we have
  $N=0$, \textit{i.e.}, $Rm\ne 0$ for any nonzero $m\in M$; again
  because $M$ is s-simple, it follows that $M = Rm$ for any nonzero
  $m\in M$.

  (iii) $\Longrightarrow$ (ii).  Taking into consideration that $M$ is a
  $dR$-semimodule, this implication can be established in the same
  fashion as the previous one.

  The implications (ii) $\Longrightarrow$ (i) and (ii)
  $\Longrightarrow$ (iii) are obvious.
\end{proof}

A congruence $\rho$ on a complete left $R$-semimodule $M$ is called
\textit{complete} iff for any two families $\{m_i \mid i\in \Omega \}$
and $ \{m_i'\mid i\in\Omega\}$ of elements of $M$, the
following implication
\begin{equation*}
  \forall i\in \Omega : m_{i}\,\rho\, m_{i}'
  \ \Longrightarrow\
  (\mbox{$\sum_{i\in \Omega}m_i$})\,\rho\, 
  (\mbox{$\sum_{i\in\Omega}m_i'$})
\end{equation*}
is true.  For any complete left $R$-semimodule $M$, the diagonal,
$\vartriangle_M$, and universal, $M^2$, congruences are obviously
complete congruences on $M$; and a complete left $R$-semimodule $M$ is
called \textit{complete congruence-simple} (\textit{cc-simple}) iff
$\vartriangle_M$ and $M^2$ are the only complete congruences on $M$.
Of course, the right-sided analog of the cc-simpleness is defined
similarly.  A left (right) $dR$-semimodule $M$ is called
\textit{simple} iff it is both s-simple and congruence-simple,
\textit{i.e.}, $\vartriangle_M$ and $ M^2$ are the only congruences on
the left (right) $R$-semimodule $M$ (and in this case, $M$ is
obviously cc-simple, too).  The following fact, illustrating natural
simple complete semimodules, deserves mentioning.

\begin{proposition}
  Let $M$ be a nonzero complete $\mathbf{B}$-semimodule.  Then, $M$ is
  a simple complete left (right) $R$-semimodule over any dense
  complete subhemiring~$R$ of the endomorphism semiring
  $\mathbf{E}_M$.  In particular, any complete lattice~$M$ is a simple
  complete left (right) ${\rm CEnd}_{\mathbf{B}}(M)$-semimodule.
\end{proposition}

\begin{proof}
  For a dense complete subhemiring $R$ of the dense, by Lemma~3.4,
  subsemiring ${\rm CEnd}_{\mathbf{B}}(M)$ of the endomorphism
  semiring $\mathbf{E}_M$, the scalar multiplication from $R\times M$
  to $M$ is defined by $fm:=f(m)$ for any $f\in R$ and $m\in M$.
  Then, it easy to see that $M$ is a complete left $R$-semimodule and,
  since $m' = e_{0,m'}(m) = e_{0,m'}m\in Rm$ for any nonzero $m\in M$
  and $m'\in M$, we have that $Rm=M$.  Hence, by Lemma~4.1, $M$ is an
  s-simple complete left $dR$-semimodule.

  Now suppose that $\rho\ne\vartriangle_M$ is a congruence on the left
  $R$-semimodule $M$ and there exist elements $x,y\in M$ such that
  $x\ne y$ and $(x,y)\in\rho$; moreover, since $\rho$ is a congruence,
  without loss of generality we may even assume that $y<x$.  Then,
  $(x,0) = (e_{y,x}(x), e_{y,x}(y)) = (e_{y,x}\,x, e_{y,x}\,y)\in\rho$
  and $(m,0) = (e_{0,m}(x), e_{0,m}(0)) = (e_{0,m}\,x,
  e_{0,m}\,0)\in\rho$ for any $m\in M$, and therefore, $\rho =M^2$.
\end{proof}

For any element $x\in R$ of a complete hemiring $R$ and any cardinal
number~$n$, there exists the ``\textit{additive} $n$\textit{-power}''
$nx$ of $x$ defined by $nx:=\sum_{i\in I}x$ with $I$ to be a set of
the cardinality~$n$, \textit{i.e.}, $|I|=n$.  From the definition of
the summation $\Sigma$ right away follow the following
facts:
\begin{enumerate}[\ \ (1)\ ]
\item $(n_1n_2)x=n_1(n_2x)$ and $(n_1+n_2)x=n_1x+n_2x$ for any
  cardinal numbers~$n_1$ and $n_2$ and $x\in R$;
\item $n(x+y)=nx+ny$ and $n(xy)=(nx)y=x(ny)$ for any cardinal number
  $n$ and $ x,y\in R$.
\end{enumerate}

A \textit{complete left (right) congruence} on a complete hemiring $R$
is a complete congruence on the complete left (right) regular
$R$-semimodule $R$; a \textit{complete congruence} on a complete
hemiring $R$ is a congruence that is simultaneously a complete left
and right congruence on $R$; a complete hemiring $R$ is called
\textit{complete congruence-simple (cc-simple)} if $\vartriangle_R$
and $R^2$ are the only complete congruences on $R$.

\begin{proposition}
  For a cc-simple hemiring $R$, the following statements hold:
  \begin{enumerate}[\ \ (i)\ ]
  \item $nx=x$ for any cardinal number $n$ and $x\in R$;
  \item $R$ is a complete lattice with respect to $\bigvee_{i\in I}m_i
    := \sum_{i\in I}m_i$ for any family $\{m_i\mid i\in I\}$ of
    elements of $R$;
  \item if $R^2\ne 0$ and $\infty := \sum_{x\in R}x$, then $\infty^2 =
    \infty$.
  \end{enumerate}
\end{proposition}

\begin{proof}
  (i).  For $a,b\in R$, we write $a\preceq b$ iff there exist a
  cardinal number $n$ and $x\in R$ such that $nb=x+a$; and we shall
  show that the relation $\sim$ on $R$, defined by
  \begin{equation*}
    a\sim b\ \Longleftrightarrow\ 
    a\preceq b\text{ and }b\preceq a\:,
  \end{equation*}
  is a complete congruence on $R$.  Indeed, $a\preceq a$ because
  $1a=a=0+a$; if $a\preceq b$ and $b\preceq c$, then for some cardinal
  numbers $m$ and $n$ and $x,y\in R$, we have $nb=x+a$ and $mc=y+b$,
  and, hence, $(nm)c=n(mc)=n(y+b)=ny+nb=ny+x+a$, and therefore,
  $a\preceq c$.  From this one may easily see that the relation~$\sim$
  is an equivalence relation on $R$.

  To show the completeness of the relation $\sim$, consider two
  families $(a_i)_{i\in I}$ and $(b_i)_{i\in I}$ of elements of $R$
  such that $a_i\preceq b_i$ for each $i\in I$.  Then, there are
  cardinal numbers $n_i$ and elements $x_i\in R$, $i\in I$, such that
  $ n_ib_i=x_i+a_i $.  Appealing to the obvious natural sense of the
  cardinal arithmetic and assuming $n:=\Sigma _{i\in I}n_i$, one has
  $nb_i = (n-n_i)b_i+n_ib_i = (n-n_i)b_i+x_i+a_i = x_i'+a_i$, where
  $x_i' := (n-n_i)b_i+x_i$, $i\in I$, and therefore, $n\big(\sum_{i\in
    I}b_i\big) = \sum_{i\in I}(nb_i) = \sum_{i\in I}(x_i'+a_i) =
  \sum_{i\in I}x_i' + \sum_{i\in I}a_i$.  The latter implies
  $\sum_{i\in I}a_i\preceq \sum_{i\in I}b_i$ and, hence, the
  completeness of $\sim$\,.

  To see that $\sim$ is a congruence on $R$, suppose that $a\preceq b$
  and $ c\preceq d$.  Then there are cardinal numbers $m$ and $n$ and
  elements $x,y\in R$ such that $nb=x+a$ and $md=y+c$, and therefore,
  $(mn)(bd) = m(nb)d = m(x+a)d = (x+a)md = (x+a)(y+c) = xy+xc+ay+ac$.
  For the latter implies $ac\preceq bd$, the relation $\sim$ is a
  congruence on $R$.

  Thus $\sim\, = \vartriangle_R$, or $\sim\, = R^2$.  In the second
  case, $a\sim 0$ for every $a\in R$ and, therefore, there exists a
  cardinal number $n$ and an element $x\in R$ such that $0=n0=x+a$.
  From this, taking into consideration the hemiring variation of
  \cite[Proposition 22.28]{golan:sata} by which the complete hemiring
  $R$ is zerosumfree, one gets $R=0$.

  If $\sim\, = \vartriangle_R$, for any $a\in R$ and nonzero cardinal
  number $n$ consider $b:=na$.  Since $na=0+b$ and $1b=b=na=(n-1)a+a$,
  one gets that $ a\sim b$, and hence, $a=b=na$.\smallskip

  (ii).  By (i), $R$ is additively idempotent, and we shall show that
  the order relation $x\le y$ if and only if $x+y=y$ turns $R$ into a
  complete lattice.  Let $(x_i)_{i\in I}$ be a family of elements in
  $R$ and $z:=\sum_{i\in I}x_i$.  If $I=\varnothing$, then $z$ is
  obviously the neutral and, hence, it is the supremum of the elements
  $x_i$, $i\in I$, so assume that $I\ne\varnothing $.  First, note
  that for all $j\in I$, we have $x_j\le z$ because $x_j+\sum_{i\in
    I}x_i = 2x_j+\sum_{i\in I\setminus\{j\}}x_i = \sum_{i\in I}x_i$.
  Now, suppose that $y\in S$ and satisfies $x_j\le y$ for all $j\in
  I$, then, for $\sum_{i\in I}x_i+y = \sum_{i\in I}x_i + \sum_{i\in
    I}y = \sum_{i\in I}(x_i+y) = \sum_{i\in I}y = y$, one gets that
  $z\le y$, and hence, $z$ is the supremum again, and, therefore, $R$
  is a complete lattice.\smallskip

  (iii).  Let $a:=\infty^2$ and $A := (a] := \{x\in R\mid x\le a\}$.
  First, note that if $x,y\in R$, then $xy\le x\infty \le \infty
  \infty=a$, so that $xy\in A$; and, hence, applying (ii), one has
  that $A$ is a complete subsemimodule of the complete regular left
  $R$-semimodule $_RR$.  Also, if $x+y\in A$ and $y\in A$, then $x\le
  x+y\le x+a\le a,$ and $x\in A$, too, i.e., $A$ is a subtractive
  ideal of $R$.

  Next, consider the Bourne congruence (see, for example,
  \cite[p.\,78]{golan:sata}) on $R$ defined by the ideal $A$: $
  x\equiv_Ay$ iff there exist elements $a,b\in A$ such that $x+a=y+b$.
  For $A$ is a complete subsemimodule of the complete left
  $R$-semimodule $_RR$, the congruence $\equiv_A$ is complete.  Thus
  $\equiv_A\,=\,\vartriangle_R$, or $\equiv _A\,=R^2$.  If
  $\equiv_A\,=\,\vartriangle_R$, then $A=\{0\}$, and hence
  $R\,R=\{0\}$, what contradicts with $R^2\neq 0$.  So,
  $\equiv_A\,=R^2$, \textit{i.e.}, $x\equiv_A0$ for all $x\in R$.
  From this and the subtractiveness of $A$, one easily concludes that
  $A=(a]=R$ and, therefore, $a=\infty$.
\end{proof}

Let $R$ be an additively idempotent complete hemiring, and $R^{\ast} =
{\rm CHom}_{\mathbf{B}}(R, \mathbf{B})$.  Then $R^{\ast}$ is a left
$dR$-semimodule with the scalar multiplication defined by $r\phi(x) :=
\phi(xr)$ for all $r,x\in R$ and $\phi\in R^{\ast}$.  As it is clear
that any homomorphism $\phi\in R^{\ast}$ is uniquely characterized by
the set $A:=\phi^{-1}(0)$, or by the element $a:=\sum_{x\in A}x\in
A=(a]:=\{x\in R\mid x\le a\}$, we will equally use $\phi_A$ and
$\phi_a$ for $\phi$.  Using these notations, introduce the cyclic left
$R$-subsemimodule $R\phi_0 := \{r\phi_0\mid r\in R\}\subseteq
R^{\ast}$ of $R^{\ast}$ generated by the homomorphism $\phi_0\in
R^{\ast}$, and note that for any $r\phi_0\in R\phi_0$ and any $x\in R$
\begin{equation*}
  r\phi_0(x)=\phi _0(xr)=
  \begin{cases}
    0 & \text{if }xr=0\:, \\ 
    1 & \text{otherwise}\:.
  \end{cases}
\end{equation*}

The following observations will prove to be useful.

\begin{lemma}
  For an additively idempotent complete hemiring $R$, the
  following statements are true:
  \begin{enumerate}[\ \ (i)\ ]
  \item $A=\phi^{-1}(0)$ is a left ideal of $R$ for any $\phi\in R\phi
    _0$;
  \item if $R$ is also cc-simple, $\psi\in R\phi_0$ and $\phi_a =
    \phi_A\in R\phi_0$, then $I:=(a\psi)^{-1}(0)$ is an ideal of $R$
    and $a\psi=0$ or $a\psi =\phi_0$.
  \end{enumerate}
\end{lemma}

\begin{proof}
  (i).  If $\phi = r\phi_0$, then $ar=0$ for any $a\in A$, and hence,
  $xar=0$ for any $x\in R$, and $xa\in A$ too.

  (ii).  By (i), to show that $I$ is an ideal, we need only to show
  that it is a right ideal.  Since $a = \sum_{r\in A}r$ it holds $x\in
  I$ iff $\big(\sum_{r\in A}r\big)\,\psi(x)=0$ iff $\psi(xr)=0$ for
  every $r\in A$.  From this and noting that, by (i), $A$ is a left
  ideal, for any $x\in I$, $ s\in R$, and $r\in A$, one has
  $0=\psi(x(sr))=\psi((xs)r)$ and, hence, $xs\in I$.

  Next, it is also easy to see that $I$ is a subtractive complete
  ideal, and the Bourne congruence $\equiv_I$ given on $R$ by
  $x\equiv_Iy$ iff there exist elements $a,b\in I$ such that
  $x+a=y+b$, is a complete one, and hence, $\equiv_I$ is $\vartriangle
  _R$ or $R^2$.  If $\equiv_I\, =\, \vartriangle_R$, then $I=\{0\}$,
  and therefore, $a\psi =\phi_0$.  In the second case, we have that
  $x\equiv _I0$ for any $x\in R$, and the subtractiveness of $I$
  implies that $I=R$, and therefore, $a\psi =0$.
\end{proof}

The following observation is important and will prove to be useful.

\begin{proposition}
  For any cc-simple hemiring $R$ with nonzero multiplication, the left
  $R$-semimodule $R\phi_0$ is a simple complete semimodule in which
  $R\phi_0$ is a complete lattice with respect to $\bigvee_{i\in
    I}\phi_i := \sum_{i\in I}\phi_i$ for any family $\{\phi_i\mid
  i\in I\}$ of homomorphisms of $R\phi _0$.
\end{proposition}

\begin{proof}
  It is clear that $R\phi_0$ is a left $dR$ -semimodule as well as a
  complete lattice with $\bigvee_{i\in I}\phi_i := \sum_{i\in
    I}\phi_i$ for any family $\{\phi_i\mid i\in I\} $ of elements of
  $R\phi_0$.  Now, let $\infty := \sum_{r\in R}$; then, applying
  Proposition 4.3\,(iii), one gets that $\infty \phi_0(\infty ) =
  \phi_0(\infty^2) = \phi_0(\infty) = 1$, and therefore,
  $R\,R\phi_0\ne 0$.

  Next, it is clear that $\phi_{\infty}\in R\phi_0$; and since $\infty
  \psi(\infty) = \psi(\infty^2) = \psi(\infty)\ne 0$ for any nonzero
  $\psi\in R\phi_0$, by Lemma~4.4, one gets that $\infty \psi =
  \phi_0$ and, hence, $R\psi = R\phi_0$ for any nonzero $\psi \in
  R\phi_0$; and therefore, by Lemma~4.1, $R\phi_0$ is s-simple.

  Thus, we have only to show that $R\phi_0$ is a cc-simple
  $R$-semimodule; and, in fact, we shall prove that $R$-semimodule
  $R\phi_0$ is even congruence-simple.  But first notice the following
  general and obvious fact: If a complete semilattice $(M,+,0_M)$ is a
  left complete $dR$-semimodule over a cc-simple hemiring $R$, then
  any nonzero complete hemiring homomorphism from $R$ to the complete
  endomorphism semiring ${\rm CEnd}_{\mathbf{B} }(M)$ is injective,
  \textit{i.e.}, a hemiring embedding; in particular, if $R\,M\ne 0$,
  the natural homomorphism from $R$ to ${\rm CEnd}_{\mathbf{B}}(M)$ is
  an embedding.

  Now, let $\rho\ne\,\vartriangle_{R\phi_0}$ be a congruence on the
  left $dR$-semimodule $R\phi_0$.  Then, there exist homomorphisms
  $\psi, \phi \in R\phi_0$ and $r_0\in R$ such that $(\psi ,\phi)\in
  \rho$ and $\psi(r_0)\ne \phi(r_0)$; and without loss of generality,
  we may assume that $\psi(r_0) = 1\in \mathbf{B}$ and $\phi(r_0) =
  0\in \mathbf{B}$.  Then using Lemma~4.4, for $A := \phi^{-1}(0)$ and
  $a := \sum_{x\in A} x\in A$ and any $r\in R$, we have
  $a\phi(r)=\phi(ra)=0$, \textit{i.e.}, $a\phi=0$.

  However, $a\psi\ne 0$.  Indeed, supposing $a\psi=0$, one gets
  $(\infty a)\psi =0$ and $(\infty r_0)\psi=0$ since $r_0\le a$, and
  hence, by Proposition 4.3\,(iii), $0=(\infty r_0)\psi(\infty ) =
  \psi(\infty^2r_0) = \psi(\infty r_0)$.  For $\phi_{\infty }\in
  R\phi_0$, by Lemma~4.4, we have $(\infty r)(x\phi_0) = \infty
  ((rx)\phi_0)\geqslant (rx)\phi_0=r(x\phi_0)$ for any $r,x\in R$, and
  therefore, taking into consideration the fact mentioned above
  together with the observation that the two multiplications on the
  left by $\infty r$ and $r\in R$, respectively, define the
  corresponding natural complete endomorphism in ${\rm
    CEnd}_{\mathbf{B} }(R\phi_0)$, one immediately gets $\infty
  r\geqslant r$ for any $r\in $ $ R $.  From the latter, we have $\psi
  (\infty r_0)\geqslant \psi (r_0) = 1$ which is a contradiction to
  $\psi (\infty r_0) = 0$.  So, $a\psi \neq 0$, and therefore, by
  Lemma~4.4, $a\psi = \phi_0$.

  Thus, the inclusion $(\psi, \phi)\in \rho$ and Lemma~4.4 imply
  that $(r(a\psi),r(a\phi)) = (r\phi_0,0)\in \rho$ for any $r\in R$,
  what, in turn, implies that $\rho$ is the universal congruence on
  $R\phi_0$, \textit{i.e.}, $\rho = (R\phi_0)^2$.
\end{proof}

In light of the proof of Proposition 4.5, it is natural to state the
following, in our view interesting, problems.

\begin{problem}
  Does there exist an s-simple, but not congruence-simple, complete
  left semimodule over a cc-simple hemiring?
\end{problem}

\begin{problem}
  For a cc-simple hemiring $R$, find and/or describe up to isomorphism
  all complete simple $R$-semimodules.
\end{problem}

Of course, it is an important and natural question whether for complete
hemirings the congruence-simpleness and cc-simpleness are the same concepts.
Our next result not only has positively answered this question, but also
describes all such hemirings.

\begin{theorem}
  For a complete hemiring $R$, the following are equivalent:
  \begin{enumerate}[\ \ (i)\ ]
  \item $R$ is congruence-simple;
  \item $R$ is cc-simple;
  \item $|R|\le 2$, or $R$ is isomorphic to a dense complete
    subhemiring $S\subseteq {\rm CEnd}_{\mathbf{B}}(M)$ of the
    complete endomorphism hemiring ${\rm CEnd}_{\mathbf{B}}(M)$ of a
    nonzero complete lattice $M$.
  \end{enumerate}
\end{theorem}

\begin{proof}
  (i) $\Longrightarrow$ (ii).  This is obvious.

  (ii) $\Longrightarrow$ (iii).  Let $R$ be a cc-simple hemiring and,
  hence, by Proposition~4.3\,(i), $R$ is additively idempotent, too.
  For the hemiring $R$, there are only the following two
  possibilities: $R$ with nonzero multiplication, \textit{i.e.},
  $R\,R\ne 0$, or $R\,R=0$.  In the second case, considering for a
  nonzero hemiring $R$ the complete hemiring homomorphism from $R$
  onto the additively idempotent two element hemiring $\mathbf{2} =
  \{0,1\}$ with the zero multiplication defined by $f(0)=0$ and
  $f(x)=1$ for all nonzero $x\in R$, we immediately obtain the
  isomorphism $R\cong \mathbf{2}$.

  So, now let $R\,R\ne 0$, and $M$ be a simple complete left
  $dR$-semimodule in which $M$ is a complete lattice with respect to
  $\bigvee_{i\in I}m_i := \sum_{i\in I}m_i$ for any family $\{m_i\mid
  i\in I\}$ of elements of $M$; an existence of such an $R$-semimodule
  $M$ is guaranteed by Proposition 4.5.  By the fact mentioned in the
  proof of Proposition 4.5, there exists the natural complete
  injection of the hemiring $R$ into the complete endomorphism
  semiring ${\rm CEnd}_{\mathbf{B}}(M)$, \textit{i.e.}, we can
  consider $R$ as a natural complete subhemiring of ${\rm
    CEnd}_{\mathbf{B}}(M)$.

  Let $l_R(x) := \{r\in R\mid rx=0\}$ be a complete left ideal of $R$
  defined for any $x\in M$.  For any family $\{m_i\mid i\in \Omega\}$
  of elements of the lattice $M$ and $r\in R$, it is clear that
  $rm_i\le \sum_{i\in \Omega}(rm_i)$ for any $i\in \Omega$ and, hence,
  $r\big(\sum_{i\in \Omega}m_i\big) = \sum_{i\in \Omega}(rm_i)=0$ iff
  $rm_i=0$ for each $i\in \Omega $; and therefore,
  $l_R\big(\Sigma_{i\in \Omega}m_i\big) = \bigcap_{i\in
    \Omega}l_R(m_i)$ for any family $\{m_i\mid i\in \Omega\}\subseteq
  R$.  As it is obvious that from $l_R(m)=l_R(m')$ follows
  $l_R(rm)=l_R(rm')$ for all $r\in R$, defining the relation $x\sim y$
  iff $l_R(x)=l_R(y)$ for $x,y\in M$, one gets a complete congruence
  on $M$ which, since $M$ is simple, coincides with the identity one
  $\vartriangle _M$.  In particular, from the latter one gets that for
  all $x,y\in M$,
  \begin{equation*}
    x\le y\ \Leftrightarrow\  x+y=y \ \Leftrightarrow\ 
    l_R(y)=l_R(x+y)=l_R(x)\cap l_R(y)\ \Leftrightarrow\ 
    l_R(y)\subseteq l_R(x)\:.
  \end{equation*}

  Now let $\infty := \sum_{m\in M}m$, and $a\in M$ such that
  $a\ne\infty$.  If $x\in M$ and $x\nleq a$, then $l_R(a)\nsubseteq
  l_R(x)$, and therefore, there is the obvious, nonzero,
  $R$-semimodule homomorphism $l_R(a)\longrightarrow M$ defined by
  $r\mapsto rx$.  For $M$ is simple, this homomorphism is surjective,
  and, hence, there exists $r_x\in l_R(a)$ such that $r_xx=\infty$.
  Letting $s:=\sum_{x\nleq a}r_x\in l_R(a)$, for any $x\in M$ we
  have
  \begin{equation*}
    sx = \begin{cases}
      s(x+a)=sx+sa=sa=0 &\text{ if }x\le a\ (x+a=a)\:,\\
      \infty &\text{ otherwise}\:. 
    \end{cases}
  \end{equation*}

  Since, by Lemma~4.1, $R\infty=M$, there exists $r\in R$ such that
  $r\infty=b$ for any $b\in M$.  From the latter, for any $x\in M$, we
  have $(rs)x = r(sx) = r0 = 0$ if $x\le a$, and $(rs)x = r(sx) =
  r\infty = b$ if $x\nleq a$, \textit{i.e.},
  \begin{equation*}
    (rs)x = \begin{cases}
      r(sx) = r0 = 0 &\text{ if }x\le a\:,\\
      r(sx) = r\infty = b &\text{ otherwise}\:.
    \end{cases}
  \end{equation*}

  Thus, for any $a,b\in M$, there exists $t\in R$ such that
  $tx=e_{a,b}(x)$ for all $x\in M$, and therefore, $R$ is a dense
  complete subhemiring of ${\rm CEnd}_{\mathbf{B}}(M)$.

  (iii) $\Longrightarrow$ (i).  This follows from Corollary 3.5.
\end{proof}

As a corollary of Theorem~4.6, we obtain the following description of left
artinian congruence-simple complete semirings.

\begin{corollary}
  A left artinian complete hemiring $R$ is congruence-simple iff
  $|R|\le 2$, or it is isomorphic to a dense complete subhemiring
  $S\subseteq {\rm CEnd}_{\mathbf{B}}(M)$ of the complete endomorphism
  hemiring ${\rm CEnd}_{ \mathbf{B}}(M)$ of a nonzero complete
  noetherian lattice $M$.
\end{corollary}

\begin{proof}
  For a nonzero complete lattice $M$, the statement immediately
  follows from Theorem~4.6 and the following observation: an
  increasing $m_1 < m_2 < \dots <m_k < \cdots$ implies
  $l_R(m_1)\supset l_R(m_2) \supset \dots \supset l_R(m_k) \supset
  \dots$ for any elements $m_1, m_2, \dots, m_k, \dots$ in $M$, for
  $e_{m_n,m_n}(m_n)=0$ and $e_{m_n,m_n}(m_{n+1})=m_n\ne 0$ for any
  $n=1,2,3,\dots$, \textit{i.e.}, $ e_{m_n,m_n}\in l_R(m_n)$ but
  $e_{m_n,m_n}\notin l_R(m_{n+1})$.
\end{proof}

The following remark is almost obvious.

\begin{remark}
  If $R$ is a finite additively idempotent hemiring, then it is a
  complete hemiring.  Indeed, $R$ is a partially ordered semiring with
  the unique partial order on $R$ defined by $r\le r'$ iff $r+r'=r'$
  (see, for example, \cite{ka:olics} and \cite[Exercise
  5.4]{kusalo:saal} for details).  With respect to this order, $ R $
  is a complete lattice and hemiring with $r\vee r':=r+r'$ and
  $r\wedge r':=\sum_{s\le r,s\le r'}s$ for all $r,r'\in R$.
\end{remark}

Taking into consideration this remark, we finish this section with another
corollary of Theorem~4.6, presenting an alternative proof of Theorem~2.2.

\begin{corollary}[{\cite[Theorem~1.7]{zumbr:cofcsswz}}]
  A proper finite hemiring $R$ is congruence-simple iff $|R|\le 2$, or
  $R$ is isomorphic to a dense subhemiring $S\subseteq\mathbf{E}_M$ of
  the endomorphism hemiring $\mathbf{E}_M$ of a finite semilattice
  $(M,+,0)$.
\end{corollary}

\begin{proof}
  $\Longrightarrow$.  By \cite[Proposition 3.1]{zumbr:cofcsswz}, $R$
  is additively idempotent, therefore, from Theorem~4.6 and
  Remark~4.8, one has that $|R|\le 2$, or $R$ is isomorphic to a dense
  complete subsemiring $S\subseteq {\rm CEnd}_{\mathbf{B}}(M)$ of the
  complete endomorphism semiring ${\rm CEnd}_{\mathbf{B}}(M)$ of a
  nonzero complete lattice $M$.  For, by Proposition~4.2, $M$ is a
  simple complete left $R$-semimodule, for any nonzero element $m\in
  M$, by Lemma~4.1, $M=Rm$, and hence, $M$ is a finite lattice.  Since
  ${\rm CEnd}_{\mathbf{B}}(M)$ is a subhemiring of $\mathbf{E}_M$,
  $S$ is a dense subhemiring of $\mathbf{E}_M$, too.

  $\Longleftarrow$.  This follows from Theorem~4.6.
\end{proof}

\section{Simpleness and Morita equivalence of semirings}

As the next observation shows, the subclass of simple semirings plays
a very special role in the class of ideal-simple semirings.  But first
recall \cite[p.\,122]{golan:sata} that a surjective homomorphism of
semirings $ f:R\longrightarrow S$ is a \textit{semiisomorphism} iff
${\rm Ker}(f)=\{0\}$; and the semiisomorphism $f$ is a \textit{strong
  semiisomorphism} if for any proper ideal $I$ of $R$, the ideal
$f(I)$ of $S$ is also proper.

\begin{proposition}
  A semiring $R$ is ideal-simple iff $R$ is a simple ring,
  or there exists a strong semiisomorphism from $R$ onto an
  additively idempotent simple semiring $S$.
\end{proposition}

\begin{proof}
  $\Longrightarrow$.  Let $\rho$ be a maximal congruence on $R$, which
  by Zorn's lemma, of course, always exists and does not contain the
  pair $(1,0)$.  Then, the factor semiring $R/\rho$ is
  congruence-simple and, since $R$ is ideal-simple, ideal-simple as
  well, \textit{i.e.}, it is simple.  For the natural surjection
  $\pi:R\longrightarrow R/\rho $, since $\pi(1_R)\notin {\rm Ker}(\pi)$, the
  ideal-simpleness of $R$ implies that ${\rm Ker}(\pi)=0$; and from
  \cite[Proposition 3.1]{zumbr:cofcsswz} and the simpleness of
  $R/\rho$, we have that $R/\rho$ is either a simple ring or an
  additively idempotent simple semiring.

  If $R/\rho$ is a simple ring, $\pi$ is an isomorphism.  Indeed, from
  the equation $\pi(a)=\pi(b)$ for some $a,b\in R$, it follows that
  there exists $c\in R$ such that $\pi(c)=-\pi (b)$ and $\pi
  (a+c)=\pi(a)+\pi(c)=\pi(b)-\pi(b)=0$, and hence, $a+c=0$ since
  ${\rm Ker}(\pi)=0$; so, for $c+b\in {\rm Ker}(\pi)$ and hence $c+b=0$ too, one
  has that $b=(a+c)+b=a+(c+b)=a$.  From the latter, we obtain that
  $\pi$ is injective and $R$ is isomorphic to the simple ring
  $R/\rho$.

  In the case when $R/\rho$ is an additively idempotent simple
  semiring, it is obvious that the semiisomorphism $\pi$ is a strong
  one.

  $\Longleftarrow$.  Taking into consideration that semiisomorphisms
  preserve ideals, this implication becomes obvious.
\end{proof}

Recall (see \cite{kat:thcos} and \cite{kn:meahcos}) that two semirings
$R$ and $S$ are said to be \textit{Morita equivalent} if the
semimodule categories $_R\mathcal{M}$ and $_{S}\mathcal{M}$ are
equivalent categories; \textit{i.e.}, there exist two (additive)
functors $F: {}_R\mathcal{M}\longrightarrow {}_{S}\mathcal{M}$ and $G:
{}_{S}\mathcal{M} \longrightarrow {}_R\mathcal{M}$, and natural
isomorphisms $\eta: GF\longrightarrow {\rm Id}_{_R\mathcal{M}}$ and
$\xi: FG\longrightarrow {\rm Id}_{_S\mathcal{M}}$.  By
\cite[Theorem~4.12]{kn:meahcos}, two semirings $R$ and $S$ are Morita
equivalent if and only if the semimodule categories $\mathcal{M}_R$
and $\mathcal{M}_S$ are equivalent categories, too.  A left semimodule
$_RP\in |_R\mathcal{M}|$ is said to be a \textit{generator} for the
category of left semimodules $_R\mathcal{M}$ if the regular semimodule
$_RR\in |_R\mathcal{M}|$ is a retract of a finite direct sum
$\bigoplus_iP$ of the semimodule $_RP$; and a left semimodule $_RP\in
|_R\mathcal{M}|$ is said to be a \textit{progenerator} for the
category of left semimodules $_R\mathcal{M}$ if it is a finitely
generated projective generator for $_R\mathcal{M}$.  By
\cite[Proposition 3.9]{kn:meahcos}, a left semimodule $_RP\in
|_R\mathcal{M}|$ is a generator for the category of left semimodules
$_R\mathcal{M}$ iff the trace ideal ${\rm tr}(P) :=
\sum_{f\in{}_R\mathcal{M}(_RP,_RR)}f(P)=R$.  Also by
\cite[Theorem~4.12]{kn:meahcos}, two semirings $R$ and $S$ are Morita
equivalent iff there exists a progenerator $_RP\in |_R\mathcal{M}|$
for $ _R\mathcal{M}$ such that the semirings $S$ and the endomorphism
semiring ${\rm End}(_RP)$ of the semimodule $_RP$ are isomorphic.
Moreover, the following observation is true.

\begin{proposition}[{cf.~\cite[Proposition 18.33]{lam:lomar}}]
  For semirings $R$ and $S$, the following statements are equivalent:
  \begin{enumerate}[\ \ (i)\ ]
  \item $R$ is Morita equivalent to $S$;
  \item $S\cong eM_n(R)e$ for some idempotent $e$ in a matrix
    semiring $M_n(R)$ \textit{such that} $M_n(R)eM_n(R)=M_n(R)$.
  \end{enumerate}
\end{proposition}

\begin{proof}
  (i) $\Longrightarrow$ (ii).  Assume $R$ is Morita equivalent to $S$.
  By \cite[Definition~4.1 and Theorem~4.12]{kn:meahcos}, there exists
  a progenerator $_RP\in |_R\mathcal{M}|$ for $_R\mathcal{M}$ such
  that $S\cong{\rm End}(_RP)$ as semirings.  Applying
  \cite[Proposition~3.1]{kn:meahcos} and without loss of generality,
  we can assume that the semimodule $_RP$ is a subsemimodule of a free
  semimodule $_RR^n$, and there exists an endomorphism $e\in {\rm
    End}(_RR^n)$ such that $e^2=e$, $ P=e(R^n)$ and $e|_{P}={\rm
    id}_{P}$.  Since $e\in{\rm End}(_RR^n)\cong M_n(R)$, one can
  consider the action of $e$ on $_RR^n$ as a right multiplication by
  some idempotent matrix $(a_{ij})\in M_n(R)$.  In the same fashion as
  it has been done in the case of the modules over rings (see, for
  example, \cite[Remark~18.10\,(D) and Exercise 2.18]{lam:lomar}), one
  may also show that ${\rm tr}(P)=\sum Ra_{ij}R$ and $rE_{ij}eE_{kl}r'
  = ra_{jk}r'E_{il}$, where $\{E_{ij}\}$ are the matrix units in
  $M_n(R)$ and $r,r'\in R$, and obtain that $M_n(R)eM_n(R)=M_n({\rm
    tr}(P))$, in the semimodule setting.  Since $_RP$ is a
  progenerator of the category of semimodules $_R\mathcal{M}$ and
  \cite[Proposition 3.9]{kn:meahcos}, ${\rm tr}(P)=R$, and hence $
  M_n(R)eM_n(R)=M_n(R)$.  We complete the proof by noting that the
  semiring homomorphism
  \begin{equation*}
    \theta: {\rm End}(_RP)\longrightarrow e{\rm End}(_RR^n)e
    \cong eM_n(R)e,
  \end{equation*}
  defined for all $f\in{\rm End}(_RP)$ by $\theta (f)=eife$ with
  $i:P\rightarrowtail R^n$ to be the natural embedding, is a semiring
  isomorphism.\smallskip

  (ii) $\Longrightarrow$ (i).  Let $S\cong eM_n(R)e$ for some
  idempotent $e$ in a matrix semiring $M_n(R)$, and $M_n(R)eM_n(R) =
  M_n(R)$.  Applying the obvious semimodule modifications of the
  well-known results for modules over rings (see, for example,
  \cite[Proposition~21.6 and Corollary~21.7]{lam:afcinr}), we have
  $S\cong eM_n(R)e\cong {\rm End}(_{M_n(R)}M_n(R)e)$.  Then, using
  that by \cite[Corollary~3.3]{kn:meahcos} $M_n(R)e$ is the projective
  left $M_n(R)$-semimodule generated by the idempotent $e$ and the
  fact mentioned in (i) that ${\rm tr}(P) = \sum Ra_{ij}R$ for any
  finitely generated projective $R$-semimodule $P$, we have that ${\rm
    tr}(_{M_n(R)}M_n(R)e) = M_n(R)eM_n(R) = M_n(R)$, and, hence, the
  semimodule $_{M_n(R)}M_n(R)e$ is a progenerator of the category of
  semimodules $_{M_n(R)}\mathcal{M}$.  From these observations and
  \cite[Definition~4.1 and Theorem~4.12]{kn:meahcos}, it follows that
  the semirings $S$ and $M_n(R)$ are Morita equivalent.  Finally,
  using the facts that by \cite[Theorem~5.14]{kat:thcos} the semirings
  $M_n(R)$ and~$R$ are also Morita equivalent, and by
  \cite[Corollary~4.4]{kn:meahcos} the Morita equivalence relation on
  the category of semirings is an equivalence relation, we conclude
  the proof.
\end{proof}

This result, in particular, motivates us to consider more carefully
the relationships between the ideal and congruence structures of
semirings $R$ and $eRe$ corresponding to idempotents $e\in R$.  So, in
the following observation, which will prove to be useful and is
interesting on its own, we consider these relationships.  Inheriting
the ring terminology (see, for example, \cite[p.\,485]{lam:lomar}), we
say that an idempotent $e\in R$ of a semiring $R$ is \textit{full} if
$ReR=R$.

\begin{proposition}
  Let $e$ be an idempotent in the semiring $R$.
  \begin{enumerate}[\ \ (i)\ ]
  \item Let $I$ be an ideal in the semiring $eRe$.  Then $e(RIR)e=I$.
    In particular, $I\longmapsto RIR$ defines an injective
    (inclusion-preserving) map from ideals of $eRe$ to those of $R$.
    This map respects multiplication of ideals, and is surjective if
    $e$ is a full idempotent.
  \item Let $\Gamma$ be a congruence on the semiring $eRe$.  Then,
    the relation $\Theta$ on $R$ for all $a,b\in R$ defined by
    \begin{equation*}
      (a,b)\in \Theta \ \Longleftrightarrow\ 
      \forall r,s\in R: (erase,erbse)\in \Gamma\:,
    \end{equation*}
    is a congruence, and $(eRe)^2\cap\Theta = \Gamma$.  In particular,
    $\Gamma \longmapsto \Theta$ defines an injective
    (inclusion-preserving) map from congruences on $eRe$ to those on
    $R$.  This map is surjective if the idempotent $e$ is
    full.
  \end{enumerate}
\end{proposition}

\begin{proof}
  (i). The proof given for rings in \cite[Theorem~21.11]{lam:afcinr}
  serves for our semiring setting as well; and just for the reader's
  convenience, we briefly sketch it here.  Namely, if $I$ is an ideal
  in $eRe$, then
  \begin{equation*}
    e(RIR)e = eR(eIe)Re = (eRe)I(eRe) = I \:,
  \end{equation*}
  and, if $I'$ is another ideal of $eRe$, then 
  \begin{equation*}
    (RIR)(RI'R) = RIRI'R = R(Ie)R(eI')R = RI(eRe)I'R = R(II')R \:.
  \end{equation*}
  If the idempotent $e$ is full, then for any ideal $J$ in $R$ and the
  ideal $eJe$ in $eRe$, we have
  \begin{equation*}
    R(eJe)R = Re(RJR)eR = (ReR)J(ReR) = RJR = J \:,
  \end{equation*}
  \textit{i.e.}, the correspondence $I\longmapsto RIR$ is surjective
  in this case.\smallskip

  (ii).  It is easy to see that the relation $\Theta$ on $R$
  corresponding to a congruence $\Gamma$ on $eRe$ is, in fact, a
  congruence on $R$.  And we shall show that $(eRe)^2\cap\Theta =
  \Gamma $.  Indeed, for any $(eae,ebe)\in \Gamma $, we have
  $(er(eae)se,er(ebe)se) = (ere(eae)ese,ere(ebe)ese)\in \Gamma$ for
  all $r,s\in R$, and hence, $\Gamma \subseteq (eRe)^2\cap\Theta$; and
  the opposite inclusion $(eRe)^2\cap\Theta \subseteq \Gamma$ is
  obvious.

  Now let an idempotent $e$ be full.  Then there exist a natural
  number $n\ge 1$ and elements $\alpha_i$, $\beta_i$ in $R$ such that
  $\sum_{i=1}^n\alpha_ie\beta_i=1$.  Let $\Pi$ be a congruence on $R$
  and $\Gamma := (eRe)^2\cap \Pi$ the congruence on $eRe$; and
  $\Theta$ the congruence on $R$ corresponding to $\Gamma$ under the
  map $\Gamma \longmapsto \Theta$.  We shall show that $\Theta = \Pi$.
  Indeed, since $\Pi$ is a congruence on $R$, for any $(a,b)\in \Pi$
  and $r,s\in R$, we always have $(erase, erbse)\in \Pi$, and hence,
  $(a,b)\in \Theta$, \textit{i.e.}, $\Pi \subseteq \Theta$.
  Conversely, for any $(a,b)\in \Theta $ and $r,s\in R$, we have
  $(erase, erbse)\in \Gamma$ and, therefore, $(erase, erbse)\in \Pi$
  for all $r, s\in R$; in particular, for all $i, j = 1, 2, \dots, n$,
  we have $(e\beta _ia\alpha_je, e\beta_ib\alpha_je)\in \Pi$ and,
  since $\Pi$ is a congruence on $R$, $(\alpha_ie\beta_i
  a\alpha_je\beta_j, \alpha_ie\beta_ib\alpha_je\beta_j)\in\Pi$, and
  therefore, $(a,b) = \sum_{j=1}^n\sum_{i=1}^n (\alpha_ie\beta_ia
  \alpha_je\beta_j, \alpha_ie\beta_ib\alpha_je\beta_j)\in\Pi$,
  \textit{i.e.}, $\Theta\subseteq\Pi$.  Thus, for a full idempotent
  $e$, the map $\Gamma \longmapsto \Theta$ is a surjective one.
\end{proof}

From Proposition 5.3 immediately follows

\begin{corollary}
  Let $e$ be a full idempotent in the semiring $R$.  Then, $R$ is
  ideal-simple (congruence-simple) iff the semiring $eRe$ is
  ideal-simple (congruence-simple).  In particular, $R$ is simple iff
  $eRe$ is simple.
\end{corollary}

We will use the following important, extending \cite[Lemma~3.1]
{bashkepka:css}, result:

\begin{proposition}[{\cite[Proposition~4.7]{knt:mosssparp}}] 
  The matrix semirings $M_n(R)$, $n\ge 2$, over a semiring $R$ are
  congruence-simple (ideal-simple) iff $R$ is congruence-simple
  (ideal-simple).  In particular, $M_n(R)$ is simple iff $R$ is
  simple.
\end{proposition}

Applying Propositions 5.2, 5.5 and Corollary~5.4, we immediately
establish that ideal-simpleness, congruence-simpleness and simpleness
are Morita invariants for semirings, namely:

\begin{theorem}
  Let $R$ and $S$ be Morita equivalent semirings.  Then, the semirings
  $R$ and $S$ are simultaneously congruence-simple (ideal-simple,
  simple).
\end{theorem}

In our next result, characterizing simple semirings with an infinite
element and complete simple semirings, we also show that these classes
of semirings are actually the same.

\begin{theorem}
  For a semiring $R$, the following conditions are equivalent:
  \begin{enumerate}[\ \ (i)\ ]
  \item $R$ is a simple semiring with an infinite element;
  \item $R$ is Morita equivalent to the Boolean semiring $\mathbf{B}$;
  \item $R\cong\mathbf{E}_M$, where $M$ is a nonzero finite
    distributive lattice;
  \item $R$ is a complete simple semiring.
  \end{enumerate}
\end{theorem}

\begin{proof}
  (i) $\Longrightarrow $ (ii).  Let $R$ be a simple semiring with the
  infinite element $\infty$.  Because $\infty$ is the additively
  absorbing element of $R$, the additive reduct $(R,+,0)$ is not a
  group and $R$ a proper semiring, and hence, by
  \cite[Proposition~3.1] {zumbr:cofcsswz}, $R$ is an additively
  idempotent semiring.  Then, one readily sees that the sets $A_r :=
  \{x\in R\mid Rx=0\}$ and $A_l := \{x\in R\mid xR=0\}$ are ideals in
  $R$, and therefore, $A_r = A_l = 0$.  The latter implies $\infty
  x\infty \ne 0$ for any nonzero $x\in R$, in particular,
  $\infty^2\neq 0$.  Since the set $(\infty ^2] := \{x\in R\mid x\le
  \infty^2\}$ is obviously a nonzero ideal in $R$, one has
  $(\infty^2]=R$ and, hence, $\infty^2 = \infty$.  Again, for, as it
  is easy to see, $( \infty x\infty ] := \{y\in R\mid y\le \infty
  x\infty \}$ is an ideal in $R$, $( \infty x\infty ] = R$ for any
  nonzero $x\in R$, and therefore, for any $x\in R$, we have $\infty
  x\infty = 0$ or $\infty x\infty = \infty$.  From this observations
  we conclude that $\infty R\infty = \{0, \infty\}\cong \mathbf{B}$.
  Then, taking into consideration that $R\infty R$ is a nonzero ideal
  of $R$, one has $R\infty R=R$ and by Proposition 5.2 gets the
  implication.\smallskip

  (ii) $\Longrightarrow$ (iii).  Since the semiring $R$ is Morita
  equivalent to the Boolean semiring $\mathbf{B}$, by
  \cite[Theorem~4.12 and Definition~4.1]{kn:meahcos} there exists a
  finitely generated projective (\textit{i.e.}, a progenerator)
  $\mathbf{B}$-semimodule $_{\mathbf{B}}M\in |_{\mathbf{B}}
  \mathcal{M}|$ such that the semirings $R$ and ${\rm
    End}(_{\mathbf{B}}M) = \mathbf{E}_M$ are isomorphic. Then, taking
  into consideration that $_{\mathbf{B}}M\in
  |_{\mathbf{B}}\mathcal{M}|$ is a finitely generated semimodule and
  \cite[Theorem~5.3]{horn-kim:cos} (see also \cite[Fact~5.9]
  {kat:thcos}), one right away concludes that $M$ is a finite
  distributive lattice.\smallskip

  (iii) $\Longrightarrow$ (iv).  This implication immediately follows
  from Theorem~3.3 and Remark~4.8.\smallskip

  (iv) $\Longrightarrow$ (i).  It is obvious since by
  \cite[Proposition~22.27] {golan:sata} any complete semiring has an
  infinite element.
\end{proof}

From this result we immediately obtain the following description of all
finite simple semirings.

\begin{corollary}
  Let $R$ be a finite semiring.  Then one of the following holds:
  \begin{enumerate}[\ \ (i)\ ]
  \item $R$ is isomorphic to a matrix semiring $M_n(F)$ for a
    finite field $F$, $n\ge 1$;
  \item $R\cong\mathbf{E}_M$, where $M$ is a nonzero finite
    distributive lattice.
  \end{enumerate}
\end{corollary}

\begin{proof}
  Indeed, if $R$ is simple, then by
  \cite[Proposition~3.1]{zumbr:cofcsswz} $R$ is either a ring or an
  additively idempotent semiring.  In the first case, one obtains
  $R\cong M_n(F)$ for some finite field $F$ by the well-known
  variation of the classical Wedderburn-Artin Theorem, characterizing
  simple artinian rings (see, for example,
  \cite[Theorem~1.3.10]{lam:afcinr}).  In the case when $R$ is a
  finite additively idempotent semiring, $R$ has an infinite element
  and the result follows from Theorem~5.7.
\end{proof}

In connection with this corollary, we wish to mention the following
remarks and, in our view interesting, problem.

\begin{remark}
  It is easy to see that a finite proper hemiring $R$ is simple iff
  $R\cong \mathbf{F}_M$ for some finite lattice $M$.  Indeed, if $R$
  is simple, then, by \cite[Theorem~1.7]{zumbr:cofcsswz}, for some
  finite lattice $M$, there exists a subhemiring $S$ of a semiring
  $\mathbf{E}_M$ such that $R\cong S$ and $\mathbf{F}_M\subseteq
  S\subseteq\mathbf{E}_M$, and since, by Lemma~3.1, $\mathbf{F}_M$ is
  an ideal of $ \mathbf{E}_M$, the ideal-simpleness of $S$ implies
  $S=\mathbf{F}_M$.  Inversely: If $R\cong \mathbf{F}_M$ for some
  finite lattice $M$, then by the same
  \cite[Theorem~1.7]{zumbr:cofcsswz} $\mathbf{F}_M$ is
  congruence-simple; and assuming that $I\subseteq \mathbf{F}_M$ is a
  nonzero ideal and $f(m)\ne 0$ for some $f\in I$ and $m\in M$, we
  have $ e_{a,b} = e_{0,b}\,f\,e_{a,m}\in I$ for all $a,b\in M$ and,
  hence, $\mathbf{F}_M\subseteq I$, and therefore, $\mathbf{F}_M=I$
  and $\mathbf{F}_M$ is ideal-simple too.  Moreover, using
  \cite[Proposition~2.3]{zumbr:cofcsswz} in the same fashion one may
  easily see that for any lattice $M$ with an absorbing (infinite)
  element, $\mathbf{F}_M$ is always a simple proper hemiring with an
  infinite element; however, the following open question is of
  interest.
\end{remark}

\begin{problem}
  Is it right that for any proper simple hemiring $R$ with the
  infinite element, there exists a lattice $M$ with the infinite
  element such that the hemirings $R$ and $\mathbf{F}_M$ are
  isomorphic?
\end{problem}

In our next observation we present a semiring analog of the well-known
``Double Centralizer Property'' of ideals of simple rings (see,
\textit{e.g.}, \cite[Theorem~1.3.11]{lam:afcinr}).

\begin{theorem}[{cf. \cite[Theorem~1.3.11]{lam:afcinr}}]
  Let $R$ be a simple semiring, and $I$ be a nonzero left ideal.  Let
  $D={\rm End}({}_RI)$ (viewed as a semiring of right operators on
  $I$).  Then the natural map $f:R\longrightarrow {\rm End}(I_D)$ is a
  semiring isomorphism; the semimodule $I\in {}_R\mathcal{M}$ is a
  generator in the category $_R\mathcal{M}$, and the semimodule
  $I\in\mathcal{M}_D$ is a finitely generated projective right
  $D$-semimodule.  Moreover, there is a nonzero idempotent $e$ of the
  matrix semiring $M_n(D)$, $n\ge 1$, such that the semirings $R$ and
  $eM_n(D)e$ are isomorphic, and the semiring $D$ is simple iff the
  left ideal $I$ is a finitely generated projective left
  $R$-semimodule.
\end{theorem}

\begin{proof}
  Since the natural map $f:R\longrightarrow {\rm End}(I_D)$ is defined
  by $r\longmapsto f(r)\in {\rm End}(I_D)$, where $f(r)(i):=ri$ for
  any $i\in I$ and $r\in R$, it is clear that $f$ is a semiring
  homomorphism from $R$ to ${\rm End}(I_D)$, and $ _RI_{D}$ is an
  $R$-$D$-bisemimodule (\cite{kat:thcos}).  For the semiring $R$ is
  congruence-simple, the map $f$ is injective and there should be only
  shown that it is surjective, too.  And, for the ideal-simpleness
  of~$R$, the latter can be established by repeating verbatim the
  scheme of the proof of Theorem~1.3.11 of \cite{lam:afcinr}.  Thus,
  $f:R\longrightarrow {\rm End}(I_D)$ is an isomorphism and, hence,
  $R\cong {\rm End}(I_D)$.

  Since $R$ is a simple semiring and $I$ is a nonzero left ideal in
  $R$, the trace ideal ${\rm tr}(_RI)$ coincides with $R$,
  \textit{i.e.}, ${\rm tr}(_RI)=R$, and by
  \cite[Proposition~3.9]{kn:meahcos} the semimodule $_RI$ is a
  generator in the category of semimodule $_R\mathcal{M}$.  This fact
  implies that for some natural number $n\ge 1$ the regular semimodule
  $_RR$ is a retract of the left $R$-semimodule $I^n$, \textit{i.e.},
  there exist $R$-homomorphisms $\alpha: {}_RI^n\longrightarrow
  {}_RR$ and $\beta: {}_RR\longrightarrow {}_RI^n$ in the category
  $_R\mathcal{M}$ such that $\alpha \beta = 1_{_RR}$.  Therefore, in
  the category $\mathcal{M}_D$ , there are the obvious
  $D$-homomorphisms $\mathcal{M}_D(\alpha, 1_{_RI})(\beta, 1_{_RI}):
  {}_R\mathcal{M}(_RI^n,_RI) \longrightarrow {}_R\mathcal{M}(_RR,_RI)$
  and $\mathcal{M}_D(\alpha, 1_{_RI}):
  {}_R\mathcal{M}(_RR,_RI)\longrightarrow {}_R\mathcal{M}(_RI^n,_RI)$
  such that $\mathcal{M}_D(\beta, 1_{_RI})\mathcal{M}_D(\alpha,
  1_{_RI}) = 1_{_R\mathcal{M}(_RR,_RI)}$, as well as
  $_R\mathcal{M}(_RI^n,_RI) \cong {\rm End}(_RI)^n=D^n$ and
  $_R\mathcal{M}(_RR,_RI)\cong I_D$.  From these observations, one may
  easily see that the semimodule $I_D\in\mathcal{M}_D$ is a retract of
  the right $D$-semimodule $D^n\in\mathcal{M}_D$, and therefore, $I_D$
  is a finitely generated projective right $D$-semimodule.  Then, in
  the same manner as in the first part of the proof of
  Proposition~5.2, one may establish that $R\cong eM_n(D)e$ for some
  nonzero idempotent $e$ of the matrix semiring $M_n(D)$.

  Now assume that the left ideal $I$ is a finitely generated
  projective left $R$-semimodule.  Then, $_RI$ is a progenerator of
  the category $_R\mathcal{M}$, and, hence, $D$ is Morita equivalent
  to $R$ and, by Theorem~5.6, $D$ is a simple semiring, too.

  Finally, suppose that $D$ is a simple semiring.  It is clear that
  each element $x\in I$ produces the endomorphism $\overline{x}\in{\rm
    End}(_RI)=D$ defined by $\overline{x}(i):=ix$ for all $i\in I$,
  and denote by $\overline{I} := \{\overline{x}\mid x\in I\}\subseteq
  D$ the set of all those endomorphisms.  For any $d\in D$, it is
  clear that $\overline{xd}=\overline{x}d$, and, hence, $\overline{I}$
  is a right ideal of $D$.  For $R$ is a simple semiring, the ideal
  $l_R(I) := \{r\in R\mid ra=0\text{ for all }a\in I\}=0$.  The latter
  implies that $I^2\ne 0$ and, hence, $\overline{I}\ne 0$, and by the
  simpleness of~$D$ one has that ${\rm tr}(\overline{I}_D) = D$, what
  in turn, taking into consideration that the $D$-homomorphism
  $\theta: I_D\longrightarrow \overline{I}_D$ defined by $\theta(x) :=
  \overline{x}$ for all $x\in I$ is obviously a surjective one, gives
  that ${\rm tr}(I_D)=D$, and therefore, by
  \cite[Proposition~3.9]{kn:meahcos}, the semimodule $I_D$ is a
  generator in the category $\mathcal{M}_D$.  Then, noting that
  $R\cong {\rm End}(I_D)$, in the similar way as it was done above for
  the semimodule~$_RI$ only now substituting it with the semimodule
  $I_D$, one obtains that $_RI$ is a finitely generated projective
  left $R$-semimodule.
\end{proof}

As a corollary of Theorem~5.10, we immediately obtain the following
description of all simple semirings having projective minimal left (right)
ideals.

\begin{theorem}
  For a semiring $R$, the following statements are equivalent:
  \begin{enumerate}[\ \ (i)\ ]
  \item $R$ is a simple semiring with a projective minimal left
    (right) ideal;
  \item $R$ is isomorphic either to a matrix semiring $M_n(F)$, $n\ge
    1$, over a division ring $F$, or to an endomorphism semiring
    $\mathbf{E}_M$ of a nonzero finite distributive lattice $M$.
  \end{enumerate}
\end{theorem}

\begin{proof}
  (i) $\Longrightarrow$ (ii).  Let $R$ be a simple semiring with a
  projective minimal left ideal $I$.  Then, by Theorem~5.10, for $D =
  {\rm End}(_RI)$ and some $n\ge 1$ and a nonzero idempotent $e$ of
  the matrix semiring $M_n(D)$, there exists a semiring isomorphism
  $R\cong eM_n(D)e$.  Actually, $D$ is a division semiring: Indeed,
  for $I$ is a minimal left ideal of the semiring $R$, it is clear
  that $f(I)=I$ for any nonzero element $f\in D$, and, hence, any
  nonzero endomorphism $f\in D$ is a surjection; from the latter and
  for $_RI$ is a projective left semimodule, it follows that there
  exists an injective endomorphism $g\in D$ such that $fg=1_I$ for a
  nonzero endomorphism $f$, and therefore, $g$ and $f$ are
  isomorphisms and $D$ is a division, of course ideal-simple,
  semiring.  From this and using Proposition~5.5, we get that the
  matrix semiring $M_n(D)$ is ideal-simple, and, hence,
  $M_n(D)eM_n(D)=M_n(D)$.  Therefore, by Proposition~5.2, $R$ is Morita
  equivalent to $D$, and, hence, by Theorem~5.6, the division semiring
  $D$ is simple, too, what, by \cite[Theorem~4.5]{knt:mosssparp},
  implies that $D$ is either a division ring or the Boolean semifield
  $\mathbf{B}$.

  If $D$ is a division ring, then, since $R$ is Morita equivalent to
  $D$, it is easy to see that $R$ is a simple artinian ring, and
  therefore, $R\cong M_n(F)$ for some division ring $F$ and $n\ge 1$.
  In the case when $D$ is the Boolean semifield $\mathbf{B}$, it is
  clear that $R$ is a finite additively idempotent simple semiring
  and, therefore, by Corollary 5.8, $R\cong \mathbf{E}_M$ for some
  nonzero finite distributive lattice $M$.\smallskip

  (ii) $\Longrightarrow$ (i).  If $R\cong M_n(F)$ for some division
  ring $F$ and $n\ge 1$, the statement is a well-known classical
  result (see, for example, \cite[Theorem~1.3.5]{lam:afcinr}).

  In the second case, from Theorem~3.3 it follows that $R$ is a finite
  additively idempotent simple semiring containing a minimal left
  ideal $I$.  Let $D:={\rm End}(_RI)$ and show that
  $D\cong\mathbf{B}$: Indeed, since $I$ a minimal left ideal of the
  finite semiring $R$, it is clear that for any nonzero element $f\in
  D$, we have $f(I)=I$, and therefore, $f$ is a surjection; also, for
  $I$ is a finite\ left ideal, $f$ is an injection, too; then, since
  it is obvious that $D$ is a proper division finite semiring, from
  \cite[Corollary 1.5.9]{hebwei:sataaics} it follows that
  $D\cong\mathbf{B}$.  From the latter and Theorem~5.10, one concludes
  that $_RI$ is a projective left $R$-semimodule.
\end{proof}

The next result provides us with a characterization of ideal-simple
semirings having either infinite elements or projective minimal left (right)
ideals.

\begin{theorem}
  For a proper semiring $R$ with either the infinite element or a
  projective minimal left (right) ideal, the following statements are
  equivalent:
  \begin{enumerate}[\ \ (i)\ ]
    \item $R$ is ideal-simple;
    \item $R$ is strongly semiisomorphic to the endomorphism semiring
    $\mathbf{E}_M$ of a nonzero finite distributive lattice $M$.
  \end{enumerate}
\end{theorem}

\begin{proof}
  (i) $\Longrightarrow$ (ii).  First consider the case of a semiring
  $R$ with the infinite element $\infty\in R$.  By Proposition~5.1
  there exists a strong semiisomorphism $\alpha $ from $R$ onto an
  additively idempotent simple semiring $S$.  For $\alpha (\infty)$
  is obviously the infinite element of $S$ and Theorem~5.7, we have
  $S\cong \mathbf{E}_M$ for some nonzero finite distributive lattice
  $M$.

  Now let $R$ be an ideal-simple semiring with a projective minimal
  left (right) ideal.  Again, by Proposition~5.1, there exists a
  strong semiisomorphism $\alpha: R\twoheadrightarrow S$ from $R$ onto
  an additively idempotent simple semiring $S$.  Let $_RI$ be a
  projective minimal left ideal of $R$.  Then, it is almost obvious
  that without loss of generality (also see, for instance,
  \cite[Proposition~2.1 and Corollary~2.3]{ilinkatsov:opsvos}), one
  may assume that $I=Re$ for some idempotent $e\in R$; and $\alpha(I)
  = S\alpha(e)$ with $\alpha(e) = (\alpha(e))^2\in S$ is also a
  minimal left ideal of $S$, which, again by
  \cite[Corollary~2.3]{ilinkatsov:opsvos}, is a projective left ideal
  of $S$.  Therefore, applying Theorem~5.11, one has $S\cong
  \mathbf{E}_M$ for some nonzero finite distributive lattice $M$.

  (ii) $\Longrightarrow$ (i).  In the both cases, this follows from
  Theorem~3.3 and Proposition~5.1.
\end{proof}

\begin{corollary}
  A semiring $R$ possessing a projective minimal left (right) ideal is
  ideal-simple iff $R$ is either isomorphic to a matrix semiring
  $M_n(F)$, ${n\ge 1}$, over a division ring $F$, or strongly
  semiisomorphic to the endomorphism semiring $\mathbf{E}_M$ of a
  nonzero finite distributive lattice $M$.
\end{corollary}

\begin{remark}
  As was shown in Examples 3.7, in general, even for finite semirings,
  the congruence-simpleness and ideal-simpleness of semirings are
  independent, different notions.  However, since for any commutative
  proper semiring $R$ there exists the surjection
  $R\twoheadrightarrow\mathbf{B}$ (see, for example,
  \cite[Fact~5.5]{kat:thcos}), it is easy to see that for finite
  commutative semirings, the concepts of the ideal-simpleness and
  congruence-simpleness are always the same, i.e., coincide.
  In light of this observation and Theorem~5.12, the following two
  open questions, in our view, are of interest.
\end{remark}

\begin{problem}
  What is the class of all semirings (in particular, finite ones) for
  which the concepts of the ideal-simpleness and congruence-simpleness
  coincide? (In other words, to describe all semirings for which the
  concepts of the ideal-simpleness and congruence-simpleness
  coincide.)
\end{problem}

\begin{problem}
  Describe the class of all proper semirings (in particular, ones
  possessing either infinite elements or projective minimal one-sided
  ideals) which are strongly semiisomorphic to the endomorphism
  semirings $\mathbf{E}_M$ of nonzero finite distributive lattices
  $M$.
\end{problem}

We finish this section considering two more applications -- to
Conjecture and Problem~3.9 of \cite{kat:thcos} and \cite{kat:ofsos},
respectively -- of Theorem~5.7.  But first recall
(\cite[Definition~3.1]{kat:tpaieosoars} or \cite{kat:thcos}) that the
tensor product bifunctor $-\otimes-: {\mathcal{M}_R\times {}_R\mathcal{M}
\longrightarrow \mathcal{M}}$ of a right semimodule $A\in
|\mathcal{M}_R|$ and a left semimodule ${B\in |_R\mathcal{M}|}$ can be
described as the factor monoid $F/\sigma$ of the free monoid ${F\in
|\mathcal{M}|}$ generated by the cartesian product $A\times B$ and
factorized with respect to the congruence $\sigma$ on $F$ generated
by ordered pairs having the form
\begin{equation*}
  \langle (a_1\!+\!a_2,b), (a_1,b)+(a_2,b)\rangle,\ 
  \langle (a,b_1\!+\!b_2), (a,b_1)+(a,b_2)\rangle,\
  \langle (ar,b),(a,rb)\rangle,
\end{equation*}
with $a,a_1,a_2\in A$, $b,b_1,b_2\in B$ and $r\in R$.

A semimodule $B\in |_R\mathcal{M}|$ is said to be \textit{flat} \cite
{kat:ofsos} iff the functor $-\otimes B: {\mathcal{M}_R\longrightarrow
\mathcal{M}}$ preserves finite limits or iff a semimodule $B$ is a
filtered (directed) colim of finitely generated free (even projective)
semimodules \cite[Theorem~2.10]{kat:ofsos}; a semiring $R$ is left
(right) \textit{perfect} \cite{kat:thcos} iff every flat left (right)
$R$-semimodule is projective.

In \cite[Corollary~5.12]{kat:thcos}, it was shown that in the class of
additively regular commutative semirings, perfect semirings are just
perfect rings, and proposed the conjecture that the same situation
would take place in the entire, not only commutative, class of
additively regular semirings.  Then this conjecture has been confirmed
for additively regular semisimple semirings in
\cite[Theorem~5.2]{knt:mosssparp}.  By using Theorem~5.7, we confirm
the conjecture for the class of simple semirings with either infinite
elements or projective minimal left (right) ideals.

\begin{theorem}
  A proper simple semiring $R$ with either the infinite element or a
  projective minimal left (right) ideal is not left (right) perfect.
\end{theorem}

\begin{proof}
  First consider the case of a semiring $R$ with the infinite element
  $\infty\in R$.  By Theorem~5.7, $R$ is Morita equivalent to the
  Boolean semiring $\mathbf{B}$.  Therefore, by
  \cite[Theorem~4.12]{kn:meahcos} the semimodule categories
  $_R\mathcal{M}$ and $_{\mathbf{B}}\mathcal{M}$ are equivalent,
  \textit{i.e.}, there exist two (additive) functors $F:
  _R\mathcal{M}\longrightarrow {}_{\mathbf{B}}\mathcal{M}$ and $G:
  _{\mathbf{B}}\mathcal{M}\longrightarrow {}_R\mathcal{M}$, and
  natural isomorphisms $\eta: GF\longrightarrow {\rm
    Id}_{_R\mathcal{M}}$ and $\xi: FG\longrightarrow {\rm
    Id}_{_{\mathbf{B}}\mathcal{M}}$.  By \cite[Lemma~4.10 and
  Proposition~5.12]{kn:meahcos}, the functors $F$ and $G$ establish
  the equivalences between the subcategories of projective and flat
  left semimodules of the categories $_R\mathcal{M}$ and
  $_{\mathbf{B}}\mathcal{M}$, respectively.  However, by \cite[Theorem
  5.11]{kat:thcos} (see also \cite[Theorem~5.2]{knt:mosssparp})
  $\mathbf{B}$ is not a perfect semiring and, hence, $R$ is not
  perfect, too.

  In the case of a semiring $R$ with a projective minimal left (right)
  ideal, from Theorem~5.11 follows that $R\cong\mathbf{E}_M$ for some
  nonzero finite distributive lattice~$M$ and, therefore, by
  Theorem~5.7, $R$ is Morita equivalent to $\mathbf{B}$, and, as it
  was shown in the first case above, $R$ is not a perfect semiring.
\end{proof}

Using Corollary~5.13 and taking into consideration, for example, 
\cite[Theorems~3.10 and~23.10]{lam:afcinr}, one has

\begin{corollary}
  A simple semiring with a projective minimal left (right) ideal $R$
  is perfect iff it is an artinian simple ring.
\end{corollary}

A semimodule $G\in |_R\mathcal{M}|$ is said to be \textit{mono-flat}
\cite {kat:ofsos} iff $\mu \otimes 1_{G}: F_1\otimes G\rightarrowtail
F\otimes G$ is a monomorphism in $\mathcal{M}$ for any monomorphism
$\mu: F_1\rightarrowtail F$ of right semimodules $F_1,F\in
|\mathcal{M}_R|$.  By \cite[Proposition~2.1 and
Theorem~2.10]{kat:ofsos}, every flat left semimodule is mono-flat, but
the converse is not true \cite[Example~3.7]{kat:ofsos} (see also
\cite[Theorem~5.19]{kn:meahcos}); and the question of describing
and/or characterizing semirings such that the concepts of
`mono-flatness' and `flatness' for semimodules over them coincide
constitutes a natural and quite interesting problem
\cite[Problem~3.9]{kat:ofsos}.  For additively regular semisimple
semirings this problem has been solved in
\cite[Theorem~5.19]{kn:meahcos}, and in our next result we positively
resolve this problem for the class of simple semirings with either
infinite elements or projective minimal left (right) ideals,
namely:

\begin{theorem}
  For left (right) $R$-semimodules over a simple semiring $R$ with
  either the infinite element or a projective minimal left (right)
  ideal, the concepts of `mono-flatness' and `flatness' are the same.
\end{theorem}

\begin{proof}
  By Theorems~5.7 and~5.11, the semiring $R$ is an artinian simple
  ring or Morita equivalent to the semiring $\mathbf{B}$.  In the
  first case, the statement is obvious.  If $R$ is Morita equivalent
  to $\mathbf{B}$, in the notations introduced in the proof of
  Theorem~5.15 and applying \cite[Proposition~5.12]{kn:meahcos}, we
  have that the functors $F$ and $G$ establish the equivalences
  between the subcategories of mono-flat and flat left semimodules of
  the categories $_R\mathcal{M}$ and $_{\mathbf{B}}\mathcal{M}$,
  respectively.  However, by \cite[Theorem~3.2]{kat:ofsos} the
  concepts of `mono-flatness' and `flatness' for
  $\mathbf{B}$-semimodules coincide and, therefore, they are the same
  for left (right) $R$-semimodules, too.
\end{proof}

\section{Simpleness of Additively Idempotent Chain Semirings}

Obviously, the additive reduct $(R,+,0)$ of an additively idempotent
semiring~$R$ in fact forms an upper semilattice and there exists the
partial ordering $\le$ on $R$ defined for any two elements $x,y\in R$
by $x\le y$ iff $x+y=y$.  If any two elements $x,y\in R$ of the poset
$(R,\le)$ are comparable, \textit{i.e.}, either $x\le y$ or $y\le x$,
the partial order relation $\le$ is said to be total, $(R,\le)$ forms
a chain, and the semiring $R$ is called an \textit{additively
  idempotent chain semiring} or, in short, \textit{aic-semiring}.  In
this final and short section, we conclude the paper considering the
congruence- and ideal-simpleness concepts in the context of artinian
simple aic-semirings and congruence-simple lattice-ordered semirings
\cite[Section~21]{golan:sata}.  In particular, we show that the only
$(\max, \cdot)$-division semirings over totally-ordered
multiplicative groups \cite[Example~4.28]{golan:sata} are left (right)
artinian ideal-simple aic-semirings.  But first let us make some
necessary observations.\medskip

\begin{lemma}
  Any left (right) artinian semiring $R$ is Dedekind-finite,
  \textit{i.e.}, ${ab=1}$ implies ${ba=1}$ for any $a,b\in R$.
\end{lemma}

\begin{proof}
  Let $R$ be a left artinian semiring, $ab=1$ for some $a,b\in R$, and
  $\ Ra^n\supseteq Ra^{n+1}$ for any $n\in \mathbb{N}$.  Then, for
  some $m\in \mathbb{N}$, $m\ge 1$, we have $Ra^{m}=Ra^{m+1}$, and,
  hence, $a^{m}=ca^{m+1}$ for some $c\in R$.  From $ab=1$ follows
  $a^{m}b^{m}=1$ , and, therefore, $1=a^{m}b^{m}=ca^{m+1}b^{m}=ca$.
  The latter implies $ b=1b=(ca)b=c(ab)=c1=c$, and, hence, $ba=1$.
\end{proof}

\begin{lemma}
  For an aic-semiring $R$, the subset $J:=\{ a\in R\mid \forall r\in
  R: {ra\ne 1}\}\subseteq R$ of the semiring $R$\ is the only maximal
  left ideal of $R$; and therefore, $J={\rm Rad}(_RR)$.
\end{lemma}

\begin{proof}
  Obviously, we need only to show that $J$ is a left ideal of $R$.
  Indeed, if $a+b\notin J$ for some $a,b\in J$, then there exists
  $r\in R$ such that $1=r(a+b)=ra+rb$ and, since $R$ is an
  aic-semiring, $ra=1$ or $rb=1$, what contradicts $a,b\in J$.
  It is clear that $ra\in R$ for any $a\in J$ and $r\in R$.
\end{proof}

In contrast to the ring case (see, \textit{e.g.},
\cite[Corollary~2.4.2]{lam:afcinr}), the radical ${\rm Rad}(_RR)$ of a
semiring $R$ in general is not an ideal of $R$ \cite[Remark,
p.\,134]{tuennam:oros}; however, for artinian aic-semirings it is not a
case, namely:

\begin{lemma}
  For a left (right) artinian aic-semiring $R$, the radical ${\rm
    Rad}(_RR)$ is an ideal of $R$, and ${\rm Rad}(_RR)={\rm
    Rad}(R_R)$.
\end{lemma}

\begin{proof}
  Indeed, if $ar\notin {\rm Rad}(_RR)$ for some $a\in{\rm Rad}(_RR)$
  and $r\in R$, then by Lemma~6.2 there exists $s\in R$ such that
  $s(ar)=1$, and hence, using Lemma~6.1, one has $(rs)a=1$.  From the
  latter, it follows that $a\notin{\rm Rad}(_RR)$ , what
  contradicts to $a\in{\rm Rad}(_RR)$.
\end{proof}

Now we are ready to describe all ideal-simple artinian
aic-semirings.

\begin{theorem}
  A left (right) artinian aic-semiring $R$ is ideal-simple
  iff it is a division aic-semiring.
\end{theorem}

\begin{proof}
  $\Longrightarrow$.  Using the ideal-simpleness of $R$, the result
  immediately follows from Lemmas 6.3 and 6.2.

  $\Longleftarrow$.  It is obvious.
\end{proof}

For a division aic-semiring $R$ is obviously zerosumfree, there always
exists the surjection $R\twoheadrightarrow\mathbf{B}$, and
therefore, from Theorem~6.4 we immediately obtain

\begin{corollary}
  A left (right) artinian aic-semiring $R$ is simple iff
  $R\cong\mathbf{B}$.
\end{corollary}

\begin{remark}
  Let $G$ be a totally-ordered multiplicative group and $R:=G\cup
  \{0\}$.  Then, extending the order on $G$ to $R$ by setting $0\le g$
  for any $g\in G$, and defining $0g=g0=0$ for all $g\in G$, one has
  that $(R,\max,\cdot)$ \cite[Section~21]{golan:sata} is a division
  aic-semiring.  And Theorem~6.4 actually says that all left (right)
  artinian ideal-simple aic-semirings can be obtained in such a
  fashion for a suitable group $G$.
\end{remark}

Recall \cite[Section 21]{golan:sata} that a semiring $R$ is
\textit{lattice-ordered} if and only if there also exists a lattice
structure $(R,\vee,\wedge)$ on $R$ such that $a+b=a\vee b$ and $ab\le
a\wedge b$ for all $a,b\in R$ with respect to the partial order
naturally induced by the lattice operations.  From the definition, it
immediately follows that a lattice-ordered semiring is an
ideal-idempotent semiring, and $a=a1\le a\wedge 1\le 1$ for all $a\in
R$, \textit{i.e.}, $1$ is the infinite element in $R$.  Form these
observations and Theorem~5.7, one readily describes all
congruence-simple lattice-ordered semirings.

\begin{theorem}
  For a lattice-ordered semiring $R$ the following statements
  are equivalent:
  \begin{enumerate}[\ \ (i)\ ]
  \item $R$ is congruence-simple;
  \item $R$ is simple;
  \item $R\cong \mathbf{B}$.
  \end{enumerate}
\end{theorem}

\begin{proof}
  (i) $\Longrightarrow$ (ii).  We need only to show that $R$ is
  ideal-simple.  Indeed, let~$I$ be a nonzero ideal of $R$.  Then, the
  Bourne relation on $R$, defined by setting $x\equiv_Iy$ iff there
  exist elements $a,b\in I$ such that $x+a=y+b$, is obviously a
  congruence on $R$, and therefore, $x\equiv_Iy$ for any $x,y\in R$.
  In particular, we have $1\equiv_I0$, \textit{i.e.}, there exist
  elements $a,b\in I$ such that $1=1+a=0+b=b\in I$, what implies that
  $I=R$.

  (ii) $\Longrightarrow$ (iii).  By Theorem~5.7, $R\cong \mathbf{E}_M$
  for some nonzero finite distributive lattice $M$.  From this and
  using the fact that $1$ is the infinite element in $R$, we get that
  $1_M:M\rightarrow M$ is the infinite element in $\mathbf{E}_M$.  On
  the other hand, putting $ \infty_M := \bigvee_{m\in M}m$, it is easy to
  see that $e_{0,\infty_M}$ is also the infinite element in
  $\mathbf{E}_M$.  Hence, $1_M = e_{0,\infty _M}$; therefore,
  $M=\{0,\infty_M\}$ and $R\cong\mathbf{E}_M\cong \mathbf{B}$.

  (iii) $\Longrightarrow$ (i).  It is obvious.
\end{proof}

\begin{corollary}
  For a lattice-ordered aic-semiring $R$ the following statements are
  equivalent:
  \begin{enumerate}[\ \ (i)\ ]
  \item $R$ is ideal-simple;
  \item $R$ is congruence-simple;
  \item $R$ is simple;
  \item $R\cong \mathbf{B}$.
  \end{enumerate}
\end{corollary}

\begin{proof}
  (i) $\Longrightarrow$ (iv).  It follows from the observation that
  the subset $I := \{{a\in R}\mid {a<1}\}\subset R$ is an ideal.

  The implication (iv) $\Longrightarrow$ (i) is obvious.

  The remaining implications follow from Theorem~6.7.
\end{proof}

\end{document}